\documentclass[12pt]{amsart}
\usepackage{amsmath,amssymb,bbm,enumerate,xspace}
\allowdisplaybreaks
\newtheorem{thm}{Theorem}
\newtheorem{lem}[thm]{Lemma}

\newtheorem{cor}[thm]{Corollary}
\theoremstyle{definition}
\newtheorem*{defn}{Definition}
\newtheorem{exa}{Example}
\newcommand{\IR}{{\mathbb R}}
\newcommand{\IC}{{\mathbb C}}
\newcommand{\IT}{{\mathbb T}}
\newcommand{\IF}{{\mathbb F}}
\newcommand{\IN}{{\mathbb N}}
\newcommand{\IZ}{{\mathbb Z}}
\newcommand{\IM}{{\mathbb M}}
\newcommand{\IB}{{\mathbb B}}
\newcommand{\cA}{{\mathcal A}}
\newcommand{\cB}{{\mathcal B}}

\newcommand{\cH}{{\mathcal H}}

\newcommand{\cQ}{{\mathcal Q}}

\newcommand{\cS}{{\mathcal S}}
\newcommand{\cT}{{\mathcal T}}
\newcommand{\fF}{{\mathfrak F}_m^d}
\newcommand{\ii}{{\mathbbm i}}
\newcommand{\ik}{{\mathbbm k}}
\newcommand{\rd}{\mathrm{d}}
\newcommand{\re}{\mathrm{e}}
\newcommand{\rh}{\mathrm{h}}
\newcommand{\rr}{\mathrm{r}}
\newcommand{\ru}{\mathrm{u}}
\newcommand{\G}{\Gamma}
\newcommand{\e}{\varepsilon}
\newcommand{\p}{\varphi}

\DeclareMathOperator*{\supp}{supp}
\DeclareMathOperator*{\lh}{span}

\DeclareMathOperator*{\diag}{diag}
\DeclareMathOperator*{\arch}{arch}
\DeclareMathOperator*{\md}{md}

\DeclareMathOperator*{\Tr}{Tr}
\DeclareMathOperator*{\tr}{tr}
\newcommand{\id}{\mathrm{id}}
\newcommand{\op}{\mathrm{op}}
\newcommand{\bubble}{0}
\newcommand{\Hil}{H}
\newcommand{\Hilk}{K}

\newcommand{\fin}{\mathrm{fin}}
\newcommand{\ip}[1]{\mathopen{\langle}#1\mathclose{\rangle}}
\newcommand{\mb}[1]{\left[ #1 \right]}
\newcommand{\questionbox}{\smash{\text{\rlap{\scriptsize\hspace*{3.4pt}\raisebox{1.6pt}{$?$}}\large$\Box$}}}

\newcommand{\shom}{$*$-homo\-mor\-phism\xspace}
\newcommand{\srep}{$*$-repre\-sen\-ta\-tion\xspace}
\newcommand{\siso}{$*$-iso\-mor\-phism\xspace}
\newcommand{\tpr}{$\tau$-pre\-serv\-ing\xspace}
\newcommand{\spc}{semi-pre-$\mathrm{C}^*$-alge\-bra\xspace}
\newcommand{\ca}{$\mathrm{C}^*$-alge\-bra\xspace}
\newcommand{\pstz}{Positiv\-stellen\-s\"atze\xspace}
\makeatletter
\newcommand{\interitemtext}[1]{%
\begin{list}{}
{\itemindent=0mm\labelsep=0mm
\labelwidth=0mm\leftmargin=0mm
\addtolength{\leftmargin}{-\@totalleftmargin}}
\item #1
\end{list}}
\makeatother

\title[About the Connes Embedding Conjecture]{About the Connes Embedding Conjecture\\
\lowercase{---algebraic approaches---}}
\author{Narutaka Ozawa}
\address{RIMS, Kyoto University, \mbox{606-8502}}
\email{narutaka@kurims.kyoto-u.ac.jp}
\thanks{Partially supported by JSPS (23540233) and by
the Danish National Research Foundation (DNRF) through the
Centre for Symmetry and Deformation}
\subjclass{16W80; 46L89, 81P15}

\keywords{Connes Embedding Conjecture, Kirchberg's Conjecture, Tsirelson's Problem,
semi-pre-C$^*$-algebras, noncommutative real algebraic geometry}
\date{\today}

\begin{document}
%
%
\maketitle
\subsection{Introduction.}
The Connes Embedding Conjecture (\cite{connes}) is considered as one of the most important
open problems in the field of operator algebras. It asserts that any finite von Neumann
algebra is approximable by matrix algebras in a suitable sense. It turns out, most notably by
Kirchberg's seminal work (\cite{kirchberg}), that the Connes Embedding Conjecture is
equivalent to a variety of other important conjectures, which touches most of the subfields
of operator algebras, and also some other branches of mathematics such as noncommutative real
algebraic geometry (\cite{schmudgen}) and quantum information theory. In this note, we look
at the algebraic aspects of this conjecture. (See \cite{bo,kirchberg,qwep} for the analytic aspects.)
This leads to a study of the $\mathrm{C}^*$-alge\-braic aspect of noncommutative real algebraic
geometry in terms of \spc{}s.
Specifically, we treat some easy parts of \pstz of Putinar (\cite{putinar}),
Helton--McCullough (\cite{hm}), and Schm\"udgen--Bakonyi--Timotin (\cite{bt}).
We then treat their tracial analogue by Klep--Schweighofer (\cite{ks}),
which is equivalent to the Connes Embedding Conjecture.
We give new proofs of Kirchberg's theorems on the tensor product
$\mathrm{C}^*\IF_d\otimes\IB(\ell_2)$ and on the equivalence between
the Connes Embedding Conjecture and Kirchberg's Conjecture.
We also look at Tsirelson's Problem in quantum information theory
(\cite{fritz,jungeetal,tsirelson}), and prove it is again equivalent to the Connes Embedding Conjecture.
This paper is an expanded lecture note for the author's lecture for ``Masterclass on sofic groups
and applications to operator algebras" (University of Copenhagen, 5--9 November 2012).
The author gratefully acknowledges the kind hospitality provided by
University of Copenhagen during his stay in Fall 2012.
He also would like to thank Professor Andreas Thom for
valuable comments on this note.

%
\subsection{Ground assumption}
We deal with unital $*$-algebras over $\ik\in\{\IC,\IR\}$, and
every algebra is assumed to be unital, unless it is clearly not so.
The unit of an algebra is simply denoted by $1$ and all homomorphisms
and representations between algebras are assumed to preserve the units.
We denote by $\ii$ the imaginary unit, and by $\lambda^*$ the complex conjugate
of $\lambda\in\IC$.
In case $\ik=\IR$, one has $\lambda^*=\lambda$ for all $\lambda\in\ik$.
\subsection{Semi-pre-$\mathrm{C}^*$-algebras}
We will give the definition and examples of \spc{}s.
Recall that a unital algebra $\cA$ is called a \emph{$*$-algebra}
if it is equipped with a map $x\mapsto x^*$ satisfying the following properties:
\begin{enumerate}[(i)]
\item
$1^*=1$ and $(x^*)^*=x$ for every $x \in \cA$;
\item
$(xy)^*=y^*x^*$ for every $x,y \in \cA$;
\item
$(\lambda x + y)^*=\lambda^*x^*+y^*$ for every $x,y \in \cA$ and $\lambda\in\ik$.
\end{enumerate}
The sets of hermitian elements and unitary (orthogonal) elements are written respectively as
\[
\cA_\rh := \{ a\in \cA : a^*=a \}
\ \mbox{ and }\
\cA_\ru := \{ u\in \cA : u^*u=1=uu^*\}.
\]
Every element $x\in\cA$ decomposes uniquely as a sum $x=a+b$ of
an hermitian element $a$ and a skew-hermitian element $b$.
The set of hermitian elements is an $\IR$-vector space.
We say a linear map $\p$ between $*$-spaces is
\emph{self-adjoint} if $\p^*=\p$, where $\p^*$ is
defined by $\p^*(x)=\p(x^*)^*$.
We call a subset $\cA_+\subset\cA_\rh$ a \emph{$*$-positive cone}
(commonly known as a \emph{quadratic module}) if it satisfies the following:
\begin{enumerate}[(i)]
\item
$\IR_{\geq0}1\subset\cA_+$ and
$\lambda a + b \in \cA_+$ for every $a,b \in \cA_+$ and $\lambda\in\IR_{\geq0}$;
\item
$x^*ax\in\cA_+$ for every $a \in \cA_+$ and $x \in \cA$.
\end{enumerate}
For $a,b\in\cA_\rh$, we write $a\le b$ if $b-a\in\cA_+$.
We say a linear map $\p$ between spaces with positivity
is \emph{positive} if it sends positive elements to positive elements
(and often it is also required self-adjoint),
and a positive linear map $\p$ is \emph{faithful} if $a\geq0$ and $\p(a)=0$ implies $a=0$.
Given a $*$-positive cone $\cA_+$, we define the $*$-subalgebra of bounded elements by
\[
\cA^{\mathrm{bdd}}=\{ x\in \cA : \exists R>0\mbox{ such that }x^*x\le R1\}.
\]
This is indeed a $*$-subalgebra of $\cA$. 
For example, if $x$ is bounded and $x^*x\le R1$, then $x^*$ is also
bounded and $xx^*\le R1$, because
\[
0 \le R^{-1}(R1-xx^*)^2 = R1-2xx^*+R^{-1}x(x^*x)x^*\le R 1- xx^*.
\]
Thus, if $\cA$ is generated (as a $*$-algebra) by $\cS$,
then $\cS\subset\cA^{\mathrm{bdd}}$ implies $\cA=\cA^{\mathrm{bdd}}$.

\begin{defn}
A unital $*$-algebra $\cA$ is called a \emph{semi-pre-$\mathrm{C}^*$-algebra}
if it comes together with a $*$-positive cone $\cA_+$ satisfying
the Combes axiom (also called the \emph{archimedean property}) that $\cA=\cA^{\mathrm{bdd}}$.
\end{defn}

Since $h\le (1+h^2)/2$ for $h\in\cA_\rh$, one has $\cA_\rh=\cA_+-\cA_+$ for
a \spc.
We define the \emph{ideal of infinitesimal elements} by
\[
I(\cA)=\{ x\in\cA : x^*x\le\e1\mbox{ for all }\e>0\}
\]
and the \emph{archimedean closure} of the $*$-positive cone $\cA_+$ (or any other cone) by
\[
\arch(\cA_+)=\{ a\in\cA_\rh : a+\e1\in\cA_+\mbox{ for all }\e>0\}.
\]
The cone $\cA_+$ is said to be \emph{archimedean closed} if $\cA_+=\arch(\cA_+)$.
A \ca $A$ is of course a \spc,
with a zero infinitesimal ideal and an archimedean closed $*$-positive cone
\[
A_+=\{ x^*x : x\in A\}.
\]
If $A\subset\IB(\Hil)$ (here $\IB(\Hil)$ denotes the \ca
of the bounded linear operators on a Hilbert space $\Hil$ over $\ik$),
then one also has
\[
A_+=\{ a\in A_\rh : \ip{a\xi,\xi}\geq0\mbox{ for all }\xi\in\Hil\}.
\]
Note that the condition $a$ being hermitian can not be dropped when $\ik=\IR$.
It will be shown (Theorem~\ref{thm:fundamental}) that if $\cA$ is a \spc,
then $\cA/I(\cA)$ is a pre-\ca with a $*$-positive
cone $\arch(\cA_+)$.

\begin{defn}
We define the \emph{universal} \emph{$\mathrm{C}^*$-algebra}
of a \spc $\cA$
as the \ca $\mathrm{C}^*_\ru(\cA)$ together with a positive \shom
$\iota\colon\cA\to\mathrm{C}^*_\ru(\cA)$ which satisfies the following properties:
$\iota(\cA)$ is dense in $\mathrm{C}^*_\ru(\cA)$ and
every positive \srep $\pi$ of $\cA$ on a Hilbert space $\Hil$ extends
to a \srep $\bar{\pi}\colon\mathrm{C}^*_\ru(\cA)\to\IB(\Hil)$,
i.e., $\pi=\bar{\pi}\circ\iota$.
In other words, $\mathrm{C}^*_\ru(\cA)$ is the separation and completion
of $\cA$ under the $\mathrm{C}^*$-semi-norm
\[
\sup\{ \|\pi(a)\|_{\IB(\Hil)} :
\pi\mbox{ a positive \srep on a Hilbert space $\Hil$}\}.
\]
(We may restrict the dimension of $\Hil$ by the cardinality of $\cA$.)
\end{defn}
We emphasize that only \emph{positive} \srep{}s are considered.
Every positive \shom between \spc{}s
extends to a positive \shom between their universal \ca{}s.
It may happen that $\cA_+=\cA_\rh$ and $\mathrm{C}^*_\ru(\cA)=\{0\}$,
which is still considered as a unital (?) \ca.
Every \shom between \ca{}s is automatically positive, has a norm-closed range,
and maps the positive cone onto the positive cone of the range.
However, this is not at all the case for \spc{}s, as we will exhibit
a prototypical example in Example~\ref{exa:grpalg}.
On the other hand, we note that if $\cA$ is a norm-dense $*$-subalgebra of
a \ca $A$ such that $\arch(\cA_+)=\cA\cap A_+$, then every
positive \srep of $\cA$ extends to a \srep of $A$,
i.e., $A=\mathrm{C}^*_\ru(\cA)$.
(Indeed, if $x\in\cA$ has $\|x\|_A<1$, then $1-x^*x\in\cA_+$ and hence
$\|\pi(x)\|<1$ for any positive \srep $\pi$ of $\cA$.)
It should be easy to see that the following examples satisfy the axiom
of \spc{}s.

\begin{exa}\label{exa:grpalg}
Let $\G$ be a discrete group and $\ik[\G]$ be its group algebra over $\ik$:
\[
(f*g)(s)=\sum_{t\in\G} f(st^{-1})g(t)
\ \mbox{ and }\
f^*(s)=f(s^{-1})^*
\ \mbox{ for }\
f,g\in\ik[\G].
\]
The canonical $*$-positive cone of $\ik[\G]$ is defined as the sums of hermitian squares,
\[
\ik[\G]_+=\{ \sum_{i=1}^n \xi_i^**\xi_i : n\in\IN,\ \xi_i\in \ik[\G]\}.
\]
Then, $\ik[\G]$ is a \spc such that $\mathrm{C}^*_\ru(\ik[\G])=\mathrm{C}^*\G$, the full group \ca of $\G$,
which is the universal \ca generated by the unitary representations of $\G$.
There is another group \ca.
Recall that the left regular representation $\lambda$ of $\G$ on $\ell_2\G$ is defined by
$\lambda(s)\delta_t=\delta_{st}$ for $s,t\in\G$, or equivalently by $\lambda(f)\xi=f*\xi$ for
$f\in\ik[\G]$ and $\xi\in\ell_2\G$. The reduced group \ca $\mathrm{C}^*_\rr\G$ of $\G$ is
the \ca obtained as the norm-closure of $\lambda(\ik[\G])$ in $\IB(\ell_2\G)$.
The group algebra $\ik[\G]$ is equipped with the corresponding $*$-positive cone
\begin{align*}
\ik_\rr[\G]_+
 &=\{ f\in\ik[\G] : \exists f_n\in\ik[\G]_+\mbox{ such that }f_n\to f\mbox{ pointwise}\}\\
 &=\{ f\in\ik[\G] : \mbox{$f$ is of positive type}\},
\end{align*}
and the resultant \spc $\ik_\rr[\G]$ satisfies
$\mathrm{C}^*_\ru(\ik_\rr[\G])=\mathrm{C}^*_\rr\G$.
Indeed, if $f\in \ik[\G]\cap\lambda^{-1}(\mathrm{C}^*_\rr\G_+)$, then for
$\xi=\lambda(f)^{1/2}\delta_1\in\ell_2\G$, one has $f=\xi^**\xi$ and
$f$ is the pointwise limit of $\xi_n^**\xi_n \in \ik[\G]_+$,
where $\xi_n\in\ik[\G]$ are such that $\|\xi_n-\xi\|_2\to0$.
On the other hand, if $f$ is of positive type (i.e., the kernel $(x,y)\mapsto f(x^{-1}y)$ is positive semidefinite),
then $f=f^*$ and $\ip{\lambda(f)\eta,\eta}\geq0$ for every $\eta\in\ell_2\G$,
which implies $\lambda(f)\in\mathrm{C}^*_\rr\G_+$.
It follows that $\ik_\rr[\G]_+=\arch(\ik[\G]_+)$ if and only if $\G$ is amenable
(see Theorem~\ref{thm:fundamental}).
\end{exa}

\begin{exa}\label{exa:cp}
The $*$-algebra $\ik[x_1,\ldots,x_d]$ of polynomials in
$d$ commuting hermitian variables $x_1,\ldots,x_d$ is
a \spc, equipped with the $*$-positive cone
\[
\ik[x_1,\ldots,x_d]_+=\mbox{$*$-positive cone generated by }\{ 1 - x_i^2 : i=1,\ldots,d\}.
\]
One has $\mathrm{C}^*_\ru(\ik[x_1,\ldots,x_d])=C([-1,1]^d)$,
the algebra of the continuous functions on $[-1,1]^d$, and $x_i$ is identified
with the $i$-th coordinate projection.
\end{exa}

\begin{exa}\label{exa:ncp}
The $*$-algebra $\ik\langle x_1,\ldots,x_d\rangle$ of polynomials in
$d$ non-commuting hermitian variables $x_1,\ldots,x_d$ is
a \spc, equipped with the $*$-positive cone
\[
\ik\langle x_1,\ldots,x_d\rangle_+=\mbox{$*$-positive cone generated by }\{ 1 - x_i^2 : i=1,\ldots,d\}.
\]
One has $\mathrm{C}^*_\ru(\ik\langle x_1,\ldots,x_d\rangle)=C([-1,1])*\cdots*C([-1,1])$,
the unital full free product of $d$-copies of $C([-1,1])$.
\end{exa}

\begin{exa}\label{exa:tensor}
Let $\cA$ and $\cB$ be \spc{}s. We denote by
$\cA\otimes\cB$ the algebraic tensor product over $\ik$.
There are two standard ways to make $\cA\otimes\cB$ into a \spc.
The first one, called the maximal tensor product and denoted by $\cA\otimes^{\max}\cB$,
is $\cA\otimes\cB$ equipped with
\[
(\cA\otimes^{\max}\cB)_+=\mbox{$*$-positive cone generated by }\{ a\otimes b : a\in\cA_+,\ b\in\cB_+\}.
\]
The second one, called the minimal tensor product and denoted by $\cA\otimes^{\min}\cB$,
is $\cA\otimes\cB$ equipped with
\[
(\cA\otimes^{\min}\cB)_+=(\cA\otimes\cB)_\rh \cap
 (\iota_A\otimes\iota_B)^{-1}((\mathrm{C}^*_\ru(\cA)\otimes_{\min}\mathrm{C}^*_\ru(\cB))_+).
\]
(See Theorem~\ref{thm:tensorize} for a ``better" description.)
One has $\mathrm{C}^*_\ru(\cA\otimes^{\alpha}\cB)=\mathrm{C}^*_\ru(\cA)\otimes_{\alpha}\mathrm{C}^*_\ru(\cB)$
for $\alpha\in\{\max,\min\}$. The right hand side is
the \ca maximal (resp.\ minimal) tensor product (see \cite{bo,pisier}).
For $\cA_1\subset\cA_2$ and $\cB_1\subset\cB_2$, one has
\[
(\cA_1\otimes^{\min}\cB_1)_+ = (\cA_1\otimes^{\min}\cB_1) \cap (\cA_2\otimes^{\min}\cB_2)_+,
\]
but the similar identity need not hold for the maximal tensor product.
\end{exa}

\begin{exa}\label{exa:freeprod}
The unital algebraic free product $\cA*\cB$
of \spc{}s $\cA$ and $\cB$, equipped with
\[
(\cA*\cB)_+=\mbox{$*$-positive cone generated by }(\cA_+\cup \cB_+),
\]
is a \spc, and
$\mathrm{C}^*_\ru(\cA*\cB)=\mathrm{C}^*_\ru(\cA)*\mathrm{C}^*_\ru(\cB)$,
the unital full free product of the \ca{}s $\mathrm{C}^*_\ru(\cA)$ and $\mathrm{C}^*_\ru(\cB)$.
\end{exa}
The following is very basic
(cf.\ \cite{cimpric} and Proposition 15 in \cite{schmudgen}).

\begin{thm}\label{thm:fundamental}
Let $\cA$ be a \spc and $\iota\colon \cA\to \mathrm{C}^*_\ru(\cA)$
be the universal \ca of $\cA$.
Then, one has the following.
\begin{enumerate}[$\bullet$]
\item
$\ker\iota=I(\cA)$, the ideal of the infinitesimal elements.
\item
$\cA_\rh\cap\iota^{-1}(\mathrm{C}^*_\ru(\cA)_+)=\arch(\cA_+)$,
the archimedean closure of $\cA_+$.
\end{enumerate}
\end{thm}

Although it follows from the above theorem, we give here a direct proof of the fact that
$\arch(\cA_+)\cap (-\arch(\cA_+)) \subset I(\cA)$.
Indeed, if $h^2\le 1$ and $-\e1 < h <\e1$ for $\e\in(0,1)$, then one has
\[
0\le (1+h)(\e-h)(1+h) = \e(1+h)^2 -h -h(2+h)h
\le (4\e + \e)1 -(2-\e)h^2,
\]
which implies $h^2 < 5\e1$.
We postpone the proof and give corollaries to this theorem.

\subsection{Positivstellens\"atze}
We give a few results which say if an element $a$ is positive
in a certain class of representations, then it is positive for an obvious reason.
Such results are referred to as ``\pstz."
Recall that a \ca $A$ is said to be \emph{residually finite dimensional}
(RFD) if finite-dimensional \srep{}s separate
the elements of $A$, i.e., $\pi(a)\geq0$ for all finite-dimensional
\srep{}s $\pi$ implies $a\geq0$ in $A$.
All abelian \ca{}s and full group \ca{}s of
residually finite amenable groups are RFD.
Moreover, it is a well-known result of Choi that the full group \ca $\mathrm{C}^*\IF_d$
of the free group $\IF_d$ of rank $d$ is RFD (see Theorem~\ref{thm:schmudgen}).
In fact, finite representations (i.e., the unitary representations $\pi$
such that $\pi(\IF_d)$ is finite) separate the elements of $\mathrm{C}^*\IF_d$ (\cite{ls}).
However, we note that the full group \ca of a residually finite group
need not be RFD (\cite{bekka}).
We also note that the unital full free products of RFD \ca{}s is
again RFD (\cite{el}). 
In particular, $\mathrm{C}^*_\ru(\ik\langle x_1,\ldots,x_d\rangle)$ is RFD.
The results mentioned here have been proven for complex \ca{}s,
but they are equally valid for real cases. See Section~\ref{sec:realvscomplex}.
Theorem~\ref{thm:fundamental}, when combined with residual finite dimensionality, immediately
implies the following \pstz (cf.\ \cite{putinar,hm}).
\begin{cor}\label{cor:poss}
The following are true.
\begin{enumerate}[$\bullet$]
\item
Let $f\in\ik[\G]_\rh$. Then, $\pi(f)\geq0$ for every unitary representation $\pi$
if and only if $f\in\arch(\ik[\G]_+)$.
\item
The full group \ca $\mathrm{C}^*\G$ of a group $\G$
is RFD if and only if the following statement holds.
If $f\in\ik[\G]_\rh$ is such that $\pi(f)\geq0$ for every finite-dimensional
unitary representation $\pi$, then $f\in\arch(\ik[\G]_+)$.
\item
Let $f\in\ik[x_1,\ldots,x_d]_\rh$. Then, $f(t_1,\ldots,t_d)\geq0$ for
all $(t_1,\ldots,t_d)\in[0,1]^d$ if and only if
$f\in\arch(\ik[x_1,\ldots,x_d]_+)$. (See Example~\ref{exa:cp}.)
\item
Let $f\in\ik\langle x_1,\ldots,x_d\rangle_\rh$. Then, $f(X_1,\ldots,X_d)\geq0$ for
all contractive hermitian matrices $X_1,\ldots,X_d$ if and only if
$f\in \arch(\ik\langle x_1,\ldots,x_d\rangle_+)$.
(See Example~\ref{exa:ncp}.)
\end{enumerate}
\end{cor}
In some cases, the $*$-positive cones are already archimedean closed.
We will see later (Theorem~\ref{thm:schmudgen}) this phenomenon for the free group algebras $\ik[\IF_d]$.
\subsection{Eidelheit--Kakutani Separation Theorem}
The most basic tool in functional analysis is the Hahn--Banach theorem.
In this note, we will need an algebraic form of it, the Eidelheit--Kakutani Separation Theorem.
We recall the \emph{algebraic topology} on an $\IR$-vector space $V$.
Let $C\subset V$ be a convex subset. An element $c\in C$ is called
an \emph{algebraic interior} point of $C$
if for every $v\in V$ there is $\e>0$ such that $c+\lambda v\in C$ for all $|\lambda|<\e$.
The convex cone $C$ is said to be \emph{algebraically solid} if
the set $C^\circ$ of algebraic interior points of $C$ is non-empty.
Notice that for every $c\in C^\circ$ and $x\in C$, one has
$\lambda c + (1-\lambda)x \in C^\circ$ for every $\lambda\in(0,1]$.
In particular, $C^{\circ\circ}=C^\circ$ for every convex subset $C$.
We can equip $V$ with a locally convex topology, called the \emph{algebraic topology},
by declaring that any convex set that coincides with its algebraic interior is open.
Then, every linear functional on $V$ is continuous with respect to the algebraic topology.
Now Hahn--Banach separation theorem reads as follows.
\begin{thm}[Eidelheit--Kakutani (\cite{barvinok})]\label{thm:hbek}
Let $V$ be an $\IR$-vector space,
$C$ an algebraically solid cone, and $v\in V\setminus C$.
Then, there is a non-zero linear functional $\p\colon V\to\IR$ such that
\[
\p(v) \le \inf_{c\in C}\p(c).
\]
In particular, $\p(v)<\p(c)$ for any algebraic interior point $c\in C$.
\end{thm}

Notice that the Combes axiom $\cA=\cA^{\mathrm{bdd}}$ is equivalent to that
the unit $1$ is an algebraic interior point of $\cA_+\subset\cA_\rh$ and
$\arch(\cA_+)$ is the algebraic closure of $\cA_+$ in $\cA_\rh$.
(This is where the Combes axiom is needed and it can be dispensed
when the cone $\cA_+$ is algebraically closed. See Section 3.4 in \cite{schmudgen}.)
Let $\cA$ be a \spc.
A unital $*$-subspace $\cS\subset\cA$ is called a \emph{semi-operator system}.
Here, a $*$-subspace is a subspace which is closed under the $*$-operation.
Existence of $1$ in $\cS$ ensures that $\cS_+=\cS \cap \cA_+$ has enough elements
to span $\cS_\rh$. A linear functional $\p\colon\cS\to\ik$
is called a \emph{state} if $\p$ is self-adjoint, positive, and $\p(1)=1$.
Note that if $\ik=\IC$, then $\cS$ is spanned by $\cS_+$ and every positive linear
functional is automatically self-adjoint. However, this is not the case when $\ik=\IR$.
In any case, every $\IR$-linear functional $\p\colon\cS_\rh\to\IR$ extends
uniquely to a self-adjoint linear functional $\p\colon\cS\to\ik$.
We write $S(\cS)$ for the set of states on $\cS$.

\begin{cor}\label{cor:hb}
Let $\cA$ be a \spc.
\begin{enumerate}[$\bullet$]
\item
Let $W\subset\cA$ be a $*$-subspace
and $v\in\cA_{\rh}\setminus (\cA_+ + W_\rh)$.
Then, there is a state $\p$ on $\cA$ such that $\p(W)=\{0\}$ and $\p(v)\le0$.
\item {\normalfont (Krein's Extension Theorem)}
Let $\cS\subset\cA$ be a semi-operator system.
Then every state on $\cS$ extends to a state on $\cA$.
\end{enumerate}
\end{cor}
\begin{proof}
Since $\cA_+ + W_\rh$ is an algebraically solid cone in $\cA_\rh$, one may find
a non-zero linear functional $\p$ on $\cA_\rh$ such that
\[
\p(v)\le\inf\{\p(c) : c\in\cA_+ + W_\rh\}.
\]
Since $\p$ is non-zero, $\p(1)>0$ and one may assume that $\p(1)=1$.
Thus the self-adjoint extension of $\p$ on $\cA$, still denoted by $\p$,
is a state such that $\p(v)\le0$ and $\p(W_\rh)=\{0\}$.
Let $x\in W$. Then, for every $\lambda\in\ik$, one has
\[
\lambda \p(x) + (\lambda\p(x))^* = \p((\lambda x)+(\lambda x)^*)=0.
\]
This implies $\p(W)=\{0\}$ in either case $\ik\in\{\IC,\IR\}$.

For the second assertion, let $\p\in S(\cS)$ be given and consider the cone
\[
C=\{ x\in \cS_\rh : \p(x)\geq0\}+\cA_+.
\]
It is not too hard to see that $C$ is an algebraically solid cone in $\cA_\rh$
and $v\notin C$ for any $v\in \cS_{\rh}$ such that $\p(v)<0$.
Hence, one may find a state $\bar{\p}$ on $\cA$
such that $\bar{\p}(C)\subset\IR_{\geq0}$.
In the same way as above, one has that $\bar{\p}$ is zero on $\ker\p$,
which means $\bar{\p}|_\cS=\p$.
\end{proof}
\subsection{GNS construction}
We recall the celebrated GNS
construction (Gel\-fand--Nai\-mark--Segal construction),
which provides \srep{}s out of states.
Let a \spc $\cA$ and a state $\p\in S(\cA)$ be given.
Then, $\cA$ is equipped with a semi-inner product
$\ip{y,x}=\p(x^*y)$, and it gives rise to a Hilbert space,
which will be denoted by $L^2(\cA,\p)$.
We denote by $\hat{x}$ the vector in $L^2(\cA,\p)$ that corresponds to $x\in\cA$.
Thus, $\ip{\hat{y},\hat{x}}=\p(x^*y)$ and $\|\hat{x}\|=\p(x^*x)^{1/2}$.
The left multiplication $x\mapsto ax$ by an element $a\in\cA$ extends
to a bounded linear operator $\pi_\p(a)$ on $L^2(\cA,\p)$ such that
$\pi_\p(a)\hat{x}=\widehat{ax}$ for $a,x\in\cA$.
(Observe that $a^*a\le R1$ implies $\|\pi_\p(a)\|^2\le R$.)
It follows that $\pi_\p\colon\cA\to\IB(L^2(\cA,\p))$ is a positive \srep of $\cA$ such that
$\ip{\pi_\p(a)\hat{1},\hat{1}}=\p(a)$.

If $\pi\colon \cA\to\IB(\Hil)$ is a positive \srep
having a unit cyclic vector $\xi$, then $\p(a)=\ip{\pi(a)\xi,\xi}$
is a state on $\cA$ and $\pi(x)\xi\mapsto\hat{x}$ extends to a
unitary isomorphism between $\Hil$ and $L^2(\cA,\p)$ which intertwines
$\pi$ and $\pi_\p$. Since every positive \srep decomposes
into a direct sum of cyclic representations, one may obtain the universal
\ca $\mathrm{C}^*_\ru(\cA)$ of $\cA$ as the closure of the image
under the positive \srep
\[
\bigoplus_{\p\in S(\cA)}\pi_\p \colon \cA\to\IB(\bigoplus_{\p\in S(\cA)}L^2(\cA,\p)).
\]
We also make an observation that $(\cA\otimes^{\min}\cB)_+$ in Example~\ref{exa:tensor}
coincides with
\[
\{ c\in (\cA\otimes\cB)_\rh : (\p\otimes\psi)(z^*cz)\geq0
\mbox{ for all }\p\in S(\cA),\ \psi\in S(\cB),\ z\in\cA\otimes\cB\}.
\]
\subsection{Real versus Complex}\label{sec:realvscomplex}
We describe here the relation between real and complex \spc{}s.
Because the majority of the researches on \ca{}s are carried out
for complex \ca{}s, we look for a method of reducing
real problems to complex problems. Suppose $\cA_{\IR}$ is a real \spc.
Then, the complexification of $\cA_{\IR}$ is the complex
\spc $\cA_{\IC}=\cA_{\IR}+\ii\cA_{\IR}$.
The $*$-algebra structure (over $\IC$) of $\cA_{\IC}$ is defined in an obvious way,
and $(\cA_{\IC})_+$ is defined to be the $*$-positive cone generated by $(\cA_{\IR})_+$:
\[
(\cA_{\IC})_+ = \{ \sum_{i=1}^n z_i^*a_iz_i : n\in\IN,\ a_i\in(\cA_{\IR})_+,\ z_i\in\cA_{\IC}\}.
\]
(This is a temporary definition, and the official one will be given later.
See Lemma~\ref{lem:cprc}.)
Note that $\cA_{\IR}\cap(\cA_{\IC})_+=(\cA_{\IR})_+$.
The complexification $\cA_{\IC}$ has an involutive and conjugate-linear $*$-automorphism
defined by $x+\ii y\mapsto x-\ii y$, $x,y\in\cA_{\IR}$.
Every complex \spc with an involutive and
conjugate-linear $*$-automorphism arises in this way.

\begin{lem}
Let $\pi_{\IR}\colon\cA_{\IR}\to\cB_{\IR}$ be a $*$-homomorphism between real \spc{}s
(resp.\ $\p_{\IR}\colon\cA_{\IR}\to\IR$ be a self-adjoint linear functional).
Then, the complexification $\pi_{\IC}\colon\cA_{\IC}\to\cB_{\IC}$
(resp.\ $\p_{\IC}\colon\cA_{\IC}\to\IC$) is positive if and only if
$\pi_{\IR}$ (resp.\ $\p_{\IR}$) is so.
\end{lem}
\begin{proof}
We only prove that $\p_{\IC}$ is positive if $\p_{\IR}$ is so. The rest is trivial.
Let $b=\sum_i z_i^*a_iz_i\in (\cA_{\IC})_+$ be arbitrary, where $a_i\in(\cA_{\IR})_+$
and $z_i=x_i+\ii y_i$. Then,
$b=\sum_i (x_i^*a_ix_i + y_i^*a_iy_i) + \ii\sum_i(x_i^*a_iy_i-y_i^*a_ix_i)$.
Since $x_i^*a_iy_i-y_i^*a_ix_i$ is skew-hermitian, one has
$\p_{\IR}(x_i^*a_iy_i-y_i^*a_ix_i)=0$, and
$\p_{\IC}(b)=\p_{\IR}(\sum_i x_i^*a_ix_i + y_i^*a_iy_i)\geq0$.
This shows $\p_{\IC}$ is positive.
\end{proof}

We note that if $\Hil_{\IC}$ denotes the complexification of
a real Hilbert space $\Hil_{\IR}$, then
$\IB(\Hil_{\IR})_{\IC}=\IB(\Hil_{\IC})$.
Thus every positive \srep of a real
\spc $\cA_{\IR}$ on $\Hil_{\IR}$ extends
to a positive \srep of its complexification $\cA_{\IC}$
on $\Hil_{\IC}$.
Conversely, if $\pi$ is a positive \srep of $\cA_{\IC}$
on a complex Hilbert space $\Hil_{\IC}$, then its restriction to $\cA_{\IR}$
is a positive \srep on the \emph{realification} of $\Hil_{\IC}$.
The realification of a complex Hilbert space $\Hil_{\IC}$ is the real Hilbert
space $\Hil_{\IC}$ equipped with the real inner product
$\ip{\eta,\xi}_{\IR}=\Re\ip{\eta,\xi}$.
Therefore, we arrive at the conclusion that
$\mathrm{C}^*_\ru(\cA_{\IR})_{\IC}=\mathrm{C}^*_\ru(\cA_{\IC})$.
We also see that $(\IR[\G])_{\IC}=\IC[\G]$,
$(\cA_{\IR}\otimes \cB_{\IR})_{\IC}=\cA_{\IC}\otimes \cB_{\IC}$,
$(\cA_{\IR}* \cB_{\IR})_{\IC}=\cA_{\IC}* \cB_{\IC}$, etc.

\subsection{Proof of Theorem~\ref{thm:fundamental}}
We only prove the first assertion of Theorem~\ref{thm:fundamental}.
The proof of the second is very similar.
We will prove a stronger assertion that
\[
\|\iota(x)\|_{\mathrm{C}^*_\ru(\cA)}=\inf\{ R>0 : R^2 1 - x^*x \in\cA_+\}.
\]
The inequality $\le$ trivially follows from the $\mathrm{C}^*$-identity.
For the converse, assume that the right hand side is non-zero, and
choose $\lambda>0$ such that $\lambda^2 1 - x^*x\notin \cA_+$.
By Corollary~\ref{cor:hb}, there is $\p\in S(\cA)$
such that $\p(\lambda^2 1-x^*x)\le0$.
Thus for the GNS representation $\pi_\p$, one has
\[
\|\pi_\p(x)\|\geq\|\pi_\p(x)\hat{1}\|=\p(x^*x)^{1/2}\geq\lambda.
\]
It follows that $\|\iota(x)\|\geq\lambda$.
\hspace*{\fill}\qedsymbol
\subsection{Trace positive elements}
Let $\cA$ be a \spc.
A state $\tau$ on $\cA$ is called a \emph{tracial state}
if $\tau(xy)=\tau(yx)$ for all $x,y\in\cA$, or equivalently if
$\tau$ is zero on the $*$-subspace $K=\lh\{ xy-yx : x,y\in\cA\}$ spanned by commutators in $\cA$.
We denote by $T(\cA)$ the set of tracial states on $\cA$ (which may be empty).
Associated with $\tau\in T(\cA)$ is a finite von Neumann algebra $(\pi_\tau(\cA)'',\tau)$,
which is the von Neumann algebra generated by $\pi_\tau(\cA)\subset\IB(L^2(\cA,\tau))$
with the faithful normal tracial state $\tau(a)=\ip{a\hat{1},\hat{1}}$
that extends the original $\tau$.
Recall that a finite von Neumann algebra is a pair $(M,\tau)$ of
a von Neumann algebra and a faithful normal tracial state $\tau$ on $M$.
The following theorem is proved in \cite{ks} for the algebra in Example~\ref{exa:ncp}
and in \cite{jpc} for the free group algebras,
but the proof equally works in the general setting.
We note that for some groups $\G$, notably for $\G=\mathrm{SL}_3(\IZ)$ (\cite{bekkas}),
it is possible to describe all the tracial states on $\ik[\G]$ .

\begin{thm}[\cite{ks}]\label{thm:ksg}
Let $\cA$ be a \spc, and $a\in\cA_\rh$.
Then, the following are equivalent.
\begin{enumerate}[$(1)$]
\item\label{ks1}
$\tau(a)\geq0$ for all $\tau\in T(\cA)$.
\item\label{ks2}
$\tau(\pi(a))\geq0$ for every finite von Neumann algebra $(M,\tau)$
and every positive \shom $\pi\colon\cA\to M$.
\item\label{ks3}
$a\in\arch(\cA_+ + K_\rh)$, where $K_\rh=K\cap\cA_\rh = \lh\{ x^*x - xx^* : x \in \cA \}$.
\end{enumerate}
\end{thm}
\begin{proof}
The equivalence $(\ref{ks1})\Leftrightarrow(\ref{ks2})$ follows from the GNS construction.
We only prove $(\ref{ks1})\Rightarrow(\ref{ks3})$, as the converse is trivial.
Suppose $a+\e1\notin \cA_+ + K_\rh$ for some $\e>0$.
Then, by Corollary~\ref{cor:hb}, there is $\tau\in S(\cA)$ such that
$\tau(K)=\{0\}$ (i.e., $\tau\in T(\cA)$) and $\tau(a)\le-\e<0$.
\end{proof}

\subsection{Connes Embedding Conjecture}\label{sec:cec}
The Connes Embedding Conjecture (CEC)
asserts that any finite von Neumann algebra $(M,\tau)$ with
separable predual is embeddable into the ultrapower $R^\omega$ of
the hyperfinite $\mathrm{II}_1$-factor $R$ (over $\ik\in\{\IC,\IR\}$).
Here an embedding means an injective $*$-homomorphism which preserves
the tracial state. We note that if $\tau$ is a tracial state on a \spc $\cA$
and $\theta\colon\cA \to N$ is a \tpr \shom into
a finite von Neumann algebra $(N,\tau)$, then $\theta$ extends to a
\tpr \siso from $\pi_\tau(\cA)''$ onto the von Neumann subalgebra generated by $\theta(\cA)$
in $N$ (which coincides with the ultraweak closure of $\theta(\cA)$).
Hence, $(M,\tau)$ satisfies CEC if there is an ultraweakly dense $*$-subalgebra
$\cA\subset M$ which has a \tpr embedding into $R^\omega$.
In particular, CEC is equivalent to that for every countably
generated \spc $\cA$ and $\tau\in T(\cA)$,
there is a \tpr \shom from $\cA$ into $R^\omega$.
We will see that this is equivalent to the tracial analogue of
\pstz in Corollary~\ref{cor:poss}.
We first state a few equivalent forms of CEC.
We denote by $\tr$ the tracial state $\frac{1}{N}\Tr$ on $\IM_N(\ik)$.

\begin{thm}\label{thm:cecequiv}
For a finite von Neumann algebra $(M,\tau)$ with separable predual,
the following are equivalent.
\begin{enumerate}[$(1)$]
\item\label{cond:acec}
$(M,\tau)$ satisfies CEC, i.e., $M\hookrightarrow R^\omega$.
\item\label{cond:acech}
Let $d\in\IN$ and $x_1,\ldots,x_d\in M$ be hermitian contractions.
Then, for every $m\in\IN$ and $\e>0$, there are $N\in\IN$ and
hermitian contractions $X_1,\ldots,X_d \in\IM_N(\ik)$ such that
\[
|\tau(x_{i_1}\cdots x_{i_k}) - \tr(X_{i_1}\cdots X_{i_k})|<\e
\]
for all $k\le m$ and $i_j\in\{1,\ldots,d\}$.
\item\label{cond:acecu}
Assume $\ik=\IC$ (or replace $M$ with its complexification in case $\ik=\IR$).
Let $d\in\IN$ and $u_1,\ldots,u_d\in M$ be unitary elements.
Then, for every $\e>0$, there are $N\in\IN$ and
unitary matrices $U_1,\ldots,U_d \in\IM_N(\IC)$ such that
\[
|\tau(u_i^*u_j) - \tr(U_i^*U_j)|<\e
\]
for all $i,j\in\{1,\ldots,d\}$.
\end{enumerate}
In particular, CEC holds true if and only if every $(M,\tau)$ satisfies
condition $(\ref{cond:acech})$ and/or $(\ref{cond:acecu})$.
\end{thm}

The equivalence $(\ref{cond:acec})\Leftrightarrow(\ref{cond:acech})$ is
a rather routine consequence of the ultraproduct construction.
For the equivalence to (\ref{cond:acecu}), see Theorem~\ref{thm:kuq}.
Note that the assumption $\ik=\IC$ in condition (\ref{cond:acecu}) is essential
because the real analogue of it is actually true (\cite{dj}).
Since any finite von Neumann algebra $M$ with separable predual
is embeddable into a $\mathrm{II}_1$-factor which is generated by
two hermitian elements (namely $(M*R) \mathbin{\bar{\otimes}} R$),
to prove CEC, it is enough to verify the conjecture (\ref{cond:acech})
for every $(M,\tau)$ and $d=2$.
We observe that a real finite von Neumann algebra $(M_{\IR},\tau_{\IR})$
is embeddable into $R_{\IR}^\omega$ (i.e., it satisfies CEC)
if and only if its complexification $(M_{\IC},\tau_{\IC})$
is embeddable into $R_{\IC}^\omega$. The ``only if" direction is trivial and
the ``if" direction follows from the real \shom
$\IM_N(\IC)\hookrightarrow\IM_{2N}(\IR)$,
$a+\ii b\mapsto\left[\begin{smallmatrix} a & b \\ -b & a\end{smallmatrix}\right]$.
A complex finite von Neumann algebra $(M,\tau)$ need not be
a complexification of a real von Neumann algebra, but $M\oplus M^{\mathrm{op}}$ is
(isomorphic to the complexification of the realification of $M$).
Therefore, $M$ satisfies CEC if and only if its realification satisfies it.

For a finite von Neumann algebra $(M,\tau)$ and $d\in\IN$, we denote by
$\cH_d(M)$ the set of those $f \in \ik\langle x_1,\ldots,x_d\rangle_\rh$
such that $\tau(f(X_1,\ldots,X_d))\geq0$ for all
hermitian contractions $X_1,\ldots,X_d\in M$.
Further, let
\[
\cH_d=\bigcap_M \cH_d(M)
\ \mbox{ and }\
\cH_d^\fin=\bigcap_N\cH_d(\IM_N(\ik))=\cH_d(R).
\]
Notice that $\cH_d=\arch(\ik\langle x_1,\ldots,x_d\rangle_+ + K_\rh)$
(see Example~\ref{exa:ncp} and Theorem~\ref{thm:ksg}).

\begin{cor}[\cite{ks}]
Let $\ik\in\{\IC,\IR\}$. Then one has the following.
\begin{enumerate}[$\bullet$]
\item
Let $(M,\tau)$ be a finite von Neumann algebra with separable predual.
Then, $M$ satisfies CEC if and only if $\cH_d^\fin\subset\cH_d(M)$ for all $d$.
\item
CEC holds true if and only if
$\cH_d^\fin=\arch(\ik\langle x_1,\ldots,x_d\rangle_+ + K_\rh)$
for all/some $d\geq2$.
\end{enumerate}
\end{cor}
\begin{proof}
It is easy to see that condition (\ref{cond:acech}) in Theorem~\ref{thm:cecequiv}
implies $\cH_d^\fin\subset\cH_d(M)$.
Conversely, suppose condition (\ref{cond:acech}) does not hold
for some $d\in\IN$, $x_1,\ldots,x_d\in M$, $m\in\IN$, and $\e>0$.
We introduce the multi-index notation. For ${\mathbf i}=(i_1,\ldots,i_k)$,
$i_j\in\{1,\ldots,d\}$ and $k\le m$, we denote $x_{\mathbf i}=x_{i_1}\cdots x_{i_k}$.
It may happen that ${\mathbf i}$ is the null string $\emptyset$ and $x_\emptyset=1$.
Then,
\[
C=\mathrm{closure}\{(\tr(X_{\mathbf i}))_{\mathbf i} : N\in\IN,\ X_1,\ldots,X_d\in\IM_N(\ik)_\rh,\ \|X_i\|\le1\}
\]
is a convex set (consider a direct sum of matrices).
Hence by Theorem~\ref{thm:hbek}, there are $\lambda\in\IR$
and $\alpha_{\mathbf i}\in\ik$ such that
\[
\Re\sum_{\mathbf i}\alpha_{\mathbf i}\tau(x_{\mathbf i}) < \lambda
\le \inf_{\gamma\in C} \Re \sum_{\mathbf i}\alpha_{\mathbf i}\gamma_{\mathbf i}.
\]
Replacing $\alpha_{\mathbf i}$ with $(\alpha_{\mathbf i}+\alpha_{{\mathbf i}^*}^*)/2$ (here ${\mathbf i}^*$
is the reverse of ${\mathbf i}$), we may omit $\Re$ from the above inequality.
Further, arranging $\alpha_{\emptyset}$, we may assume $\lambda=0$.
Thus $f=\sum_{\mathbf i}\alpha_{\mathbf i}x_{\mathbf i}$ belongs to $\cH_d^\fin$, but not to $\cH_d(M)$.
This completes the proof of the first half.
The second half follows from this and Theorem~\ref{thm:ksg}.
\end{proof}

An analogue to the above also holds for $\IC[\IF_d]$.
\begin{cor}[\cite{jpc}]
Let $\ik=\IC$. The following holds.
\begin{enumerate}[$\bullet$]
\item
Let $(M,\tau)$ be a finite von Neumann algebra with separable predual.
Then, $M$ satisfies CEC if and only if the following holds true:
If $d\in\IN$ and $\alpha\in\IM_d(\IC)_\rh$ satisfies that
$\tr(\sum \alpha_{i,j}U_i^*U_j)\geq0$ for every $N\in\IN$ and $U_1,\ldots,U_d\in\IM_N(\IC)_\ru$,
then it satisfies $\tau(\sum \alpha_{i,j}u_i^*u_j)\geq0$ for every unitary elements $u_1,\ldots,u_d\in M$.
\item
CEC holds true if and only if for every $d\in\IN$ and $\alpha\in\IM_d(\IC)_\rh$ the following holds true:
If $\tr(\sum \alpha_{i,j}U_i^*U_j)\geq0$ for every $N\in\IN$ and $U_1,\ldots,U_d\in\IM_N(\IC)_\ru$,
then $\sum \alpha_{i,j}s_i^*s_j\in \arch(\IC[\IF_d]_+ + K_\rh)$, where $s_1,\ldots,s_d$ are the
free generators of $\IF_d$.
\end{enumerate}
\end{cor}
\subsection{Matrix algebras over \spc{}s}
We describe here how to make the $n\times n$ matrix algebra
$\IM_n(\cA)$ over a \spc $\cA$
into a \spc. We note that
$x^*=[x_{j,i}^*]_{i,j}$ for $x=[x_{i,j}]_{i,j}\in\IM_n(\cA)$.
We often identify $\IM_n(\cA)$ with $\IM_n(\ik)\otimes\cA$.
There are two natural choices of the $*$-positive cone for $\IM_n(\cA)$.
The first and the official one is
\begin{align*}
\IM_n(\cA)_+ &= \mbox{$*$-positive cone generated by }\{\diag(a_1,\ldots,a_n) : a_i\in\cA_+\}\\
&=\{ [\sum_{k=1}^m x_{k,i}^*a_kx_{k,j}]_{i,j} : m\in\IN,\ a_k\in\cA_+,\, [x_{k,i}]_{k,i}\in\IM_{m,n}(\cA)\}.
\end{align*}
The second one is larger:
\[
\IM_n(\cA)_+' = \{ [a_{i,j}]\in\IM_n(\cA)_\rh : \p(\sum_{i,j} x_i^*a_{i,j}x_j)\geq0
\mbox{ for all }x_1,\ldots,x_n\in\cA,\ \p\in S(\cA)\}.
\]
(One may notice that the first one is $\IM_n(\ik)\otimes^{\max}\cA$ and
the second is $\IM_n(\ik)\otimes^{\min}\cA$.)
In case where $\cA$ is a \ca, these cones coincide and are archimedean closed.
In either definition, one has $\mathrm{C}^*_\ru(\IM_n(\cA)) = \IM_n(\mathrm{C}^*_\ru(\cA))$,
because every \srep of $\IM_n(\cA)$ is of the form
$\id\otimes\pi\colon\IM_n(\cA)\to \IM_n(\IB(\Hil))\cong\IB(\ell_2^n\otimes\Hil)$.
Therefore, by Theorem~\ref{thm:fundamental},
we arrive at the following conclusion.
\begin{lem}\label{lem:matspc}
For every \spc $\cA$ and every $n$, one has
\[
\mathrm{C}^*_\ru(\IM_n(\cA)) = \IM_n(\mathrm{C}^*_\ru(\cA))
\ \mbox{ and }\
\arch(\IM_n(\cA)_+)=\IM_n(\cA)_+'.
\]
\end{lem}
\subsection{Completely positive maps}
Completely positive maps are essential tools in study of positivity and tensor products.
A linear map $\p\colon\cA\to\cB$ between
\spc{}s is said to be \emph{completely positive}
(c.p.) if it is self-adjoint and
$\id\otimes\p\colon \IM_n(\cA)\to\IM_n(\cB)$ is positive for every $n$.
It is not too hard to see that states and positive \shom{}s are c.p.

We take back the definition of the complexification $\cA_{\IC}$
of a real \spc $\cA_{\IR}$ and redefine
\[
(\cA_{\IC})_+=\{ x + \ii y :
x\in(\cA_{\IR})_+,\ y=-y^*\in\cA_{\IR}\mbox{ such that }
\left[\begin{smallmatrix} x & y \\ -y & x\end{smallmatrix}\right]\geq0\}.
\]
Since this coincides with the previous definition in case of \ca{}s, they also coincide
for \spc{}s modulo archimedean closures. The reason of this awkward
replacement is to assure the following lemma holds true.

\begin{lem}\label{lem:cprc}
Let $\p_{\IR}\colon \cA_{\IR}\to\cB_{\IR}$ be a self-adjoint map between
real \spc{}s. Then, its complexification $\p_{\IC}\colon \cA_{\IC}\to\cB_{\IC}$
is c.p.\ if and only if $\p_{\IR}$ is so.
\end{lem}

The GNS construction for states generalizes to
the Stinespring Dilation Theorem for c.p.\ maps.
\begin{thm}\label{thm:stinespring}
Let $\cA$ be a \spc and
$\p\colon\cA\to\IB(\Hil)$ be a c.p.\ map,
then there are a Hilbert space $\hat{\Hil}$, a positive \srep
$\pi\colon\cA\to\IB(\hat{\Hil})$, and $V\in\IB(\Hil,\hat{\Hil})$ such that
\[
\p(a)=V^*\pi(a)V.
\]
for $a\in\cA$. In particular, if $\p$ is unital, then $V$ is an isometry.
\end{thm}
\begin{proof}
We introduce the semi-inner product on the algebraic tensor product $\cA\otimes\Hil$
by $\ip{y\otimes\eta,x\otimes\xi}=\ip{\p(x^*y)\eta,\xi}$ (and extended by linearity).
That
\[
\ip{\sum_j x_j\otimes\xi_j,\sum_i x_i\otimes\xi_i}
=\ip{ [\p(x_i^*x_j)]_{i,j}\underline{\xi},\underline{\xi}}\geq0,
\mbox{ where }\underline{\xi}=(\xi_1,\ldots,\xi_n)^T,
\]
follows from the fact $[x_i^*x_j]_{i,j}\in\IM_n(\cA)_+$.
Now, $\cA\otimes\Hil$ gives rise to a Hilbert space $\hat{\Hil}$
and $V\in\IB(\Hil,\hat{\Hil})$, $\xi\mapsto1\otimes\xi$.
The \srep $\pi\colon\cA\to \IB(\hat{\Hil})$,
$\pi(a)(x\otimes\xi)=ax\otimes\xi$, is positive and satisfies $V^*\pi(a)V=\p(a)$.
\end{proof}

Let $\G$ be a discrete group. We say an operator-valued function $f\colon\G\to\IB(\Hil)$
is of \emph{positive type} if $[f(x^{-1}y)]_{x,y\in E}\in\IB(\ell_2 E)\otimes\IB(\Hil)$
is positive (and self-adjoint) for every finite subset $E\subset \G$.
It is not too hard to see the following holds.
\begin{cor}\label{cor:posdeffunct}
For $f\colon\G\to\IB(\Hil)$, the following are equivalent.
\begin{enumerate}[$(1)$]
\item
$f$ is of positive type.
\item
$\p_f\colon \ik[\G]\ni g\mapsto\sum_t g(t)f(t)\in\IB(\Hil)$ is c.p.
\item
There are a unitary representation $\pi$ of $\G$ on $\hat{\Hil}$ and
$V\in \IB(\Hil,\hat{\Hil})$ such that $f(t)=V^*\pi(t)V$ for $t\in\G$.
\end{enumerate}
\end{cor}

We denote by $S_n(\cA)$ the set of unital c.p.\ maps from $\cA$ into $\IM_n(\ik)$.

\begin{thm}\label{thm:tensorize}
Let $\cA_i$ and $\cB_i$ be \spc{}s and
$\p_i\colon\cA_i\to\cB_i$ be c.p.\ maps.
Then, $\p_1\otimes\p_2\colon \cA_1\otimes\cA_2\to\cB_1\otimes\cB_2$
is c.p.\ with respect to each of $\max$-$\max$ and $\min$-$\min$.
Moreover, $(\cA\otimes^{\min}\cB)_+$ coincides with
\[
\{ c\in(A\otimes B)_\rh :
 (\p\otimes\psi)(c)\geq0\mbox{ for all }m,n\in\IN,\ \p\in S_m(\cA),\ \psi\in S_n(\cB)\}.
\]
\end{thm}
\begin{proof}
The verification of $\max$-$\max$ case is routine (although a bit painful).
The case for $\min$-$\min$ follows from Theorem~\ref{thm:stinespring}.
We omit the details.
\end{proof}

The u.c.p.\ (unital and c.p.) maps are incorporated into the free product as well.
\begin{thm}[\cite{boca}]\label{thm:boca}
Let $\cA_i$ be \spc{}s and
$\cA=\cA_1\ast\cA_2$ be the unital free product.
Let $\psi_i\in S(\cA_i)$ be fixed and $\cA_i^\bubble=\ker\psi_i$ so that
\[
\cA=\lh\left(\ik1\cup\bigcup_n\{ x_1\cdots x_n :
  x_j\in \cA_{i_j}^\bubble,\ i_1\neq i_2,\ i_2\neq i_3,\ldots,i_{n-1}\neq i_n\}\right).
\]
Let $\p_i\colon\cA_i\to\IB(\Hil)$ be u.c.p.\ maps.
Then, the free product map $\p\colon\cA\to\IB(\Hil)$ defined by
\[
\p( x_1\cdots x_n ) = \p_{i_1}(x_1)\cdots\p_{i_n}(x_n)
\]
is a u.c.p.\ map. In particular, $\p_i$ has a joint u.c.p.\ extension $\p$.
\end{thm}
\begin{proof}
The proof of this theorem is too complicated to reproduce it here, but we prove
existence of a joint u.c.p.\ extension, which is essentially the only part that will be used in this note.
By Theorem~\ref{thm:stinespring}, each u.c.p.\ map $\p_i$ extends to a positive
\srep $\sigma_i$ on $\Hil\oplus\Hil_i$. By inflating $\Hil_i$ if necessary,
one may assume that there is a positive \srep{} $\sigma_{j,i}$ of $\cA_j$ on $\Hil_i$ for
$j\neq i$. We rewrite $\sigma_i$ by $\sigma_{i,i}$ and consider
the \srep{}s $\pi_i=\bigoplus_j\sigma_{i,j}$ of $\cA_i$ on
$\hat{\Hil}=\Hil\oplus\bigoplus_j\Hil_j$.
They extend to a positive \srep $\pi$ of $\cA$ such that
$P_{\Hil}\pi(a)|_{\Hil}=\p_i(a)$ for each $i$ and $a\in\cA_i\subset\cA$.
\end{proof}

Recall that a semi-operator system is a unital $*$-subspace $\cS$ of a
\spc $\cA$. It is equipped with
the matricial positive cone $\IM_n(\cS)_+:=\IM_n(\cS)_\rh\cap\IM_n(\cA)_+$ for each $n\in\IN$.
Thus the notion of complete positivity carries over to maps between semi-operator systems.
We introduce a convenient tool about the one-to-one correspondence between linear maps
$\p\colon\cS\to\IM_n(\ik)$ and $\tilde{\p}\colon\IM_n(\cS)\to\ik$, given by
\[
\tilde{\p}([x_{i,j}])=\sum_{i,j}\p(x_{i,j})_{i,j}
\ \mbox{ and }\
\p(x)=[\tilde{\p}(e_{i,j}\otimes x)]_{i,j}.
\]

\begin{lem}\label{lem:corr}
For a semi-operator system $\cS\subset\cA$, one has the following.
\begin{enumerate}[$\bullet$]
\item
Under the above one-to-one correspondence,
$\p$ is c.p.\ if and only if $\tilde{\p}$ is self-adjoint and positive.
\item
Let $\psi\colon\cS\to\IB(\Hil)$ be a non-zero c.p.\ map.
Then, for the support projection $p$ of $\psi(1)$, one has
$\psi(x)=p\psi(x)p$ for every $x\in\cS$, and
there is a unique u.c.p.\ map $\psi'\colon\cS\to\IB(p\Hil)$
such that $\psi(x)=\psi(1)^{1/2}\psi'(x)\psi(1)^{1/2}$.
Moreover, for every $x\in\IM_n(\ik)\otimes\cS$
one has $(\id\otimes\psi)(x)\geq0$ if and only if $(\id\otimes\psi')(x)\geq0$.
\end{enumerate}
\end{lem}
\begin{proof}
It is routine to see that $\p$ is self-adjoint if and only if $\tilde{\p}$ is.
Now, for $\zeta=\sum_i \delta_i\otimes\delta_i\in\ell_2^n\otimes\ell_2^n$,
one has $\tilde{\p}(x)=\ip{(\id\otimes\p)(x)\zeta,\zeta}$.
This proves $\tilde{\p}$ is positive
if $\id\otimes\p\colon \IM_n(\cS)\to\IM_n(\IM_n(\ik))$ is positive.
The converse follows from the identity
$\p(x)=(\id\otimes\tilde{\p})(\sum e_{i,j}\otimes e_{i,j}\otimes x)$
and the fact that every self-adjoint positive linear functional is c.p.

For the second assertion, we first observe that for
$s,t\in\IB(\Hil)$ one has
$\left[\begin{smallmatrix} s & t^* \\ t & 1\end{smallmatrix}\right]\geq0$
if and only if $t^*t\le s$ (see Lemma 3.1 in \cite{paulsen}).
Let $x \in \cS\subset\cA$ be an element such that $xx^*\le R^21$.
Then,
\[
\begin{bmatrix} R & x^* \\ x & R\end{bmatrix}
=\frac{1}{R}\begin{bmatrix} R \\ x \end{bmatrix}
 \begin{bmatrix} R & x^* \end{bmatrix}
 +\frac{1}{R}\begin{bmatrix} 0 \\ 1 \end{bmatrix}
  (R^2-xx^*)
 \begin{bmatrix} 0 & 1 \end{bmatrix}\in(\IM_2(\ik)\otimes\cA)_+.
\]
It follows that $\left[\begin{smallmatrix} R\psi(1) & \psi(x)^* \\ \psi(x) & R\psi(1)\end{smallmatrix}\right]\geq0$.
This implies that $\psi(x)p=\psi(x)$ for the support projection $p$ of $\psi(1)$ and
\[
\|(\psi(1)+\e1)^{-1/2}\psi(x)(\psi(1)+\e1)^{-1/2}\|\le R
\]
for every $\e>0$.
We define $\psi'\colon\cS\to\IB(p\Hil)$ to be a point-ultraweak limit point of
the maps
$(\psi(1)+\e1)^{-1/2}\psi(\,\cdot\,)(\psi(1)+\e1)^{-1/2}$ as $\e\searrow0$.
(This actually converges in norm.)
This completes the proof.
\end{proof}

\begin{thm}[Arveson's Extension Theorem]\label{thm:arveson}
Let $\cA$ be a \spc and
$\cS\subset\cA$ be a semi-operator system.
Then, every c.p.\ map $\p\colon\cS\to\IB(\Hil)$ extends to
a c.p.\ map $\bar{\p}\colon\cA\to\IB(\Hil)$.
\end{thm}
\begin{proof}
We first deal with the case $\Hil=\ell_2^n$. In this case, by Lemma~\ref{lem:corr},
the c.p.\ map $\p\colon\cS\to\IM_n(\ik)$ corresponds to
a positive self-adjoint linear functional $\tilde{\p}$ on $\IM_n(\cS)$.
By Corollary~\ref{cor:hb}, it extends to a positive self-adjoint linear functional
on $\IM_n(\cA)$. Using the correspondence again, one obtains a c.p.\ extension
$\bar{\p}$ on $\cA$.

Now, let $\Hil$ be arbitrary. We denote by $\{\Hil_i\}_{i\in I}$ the
directed set of the finite-dimensional subspaces of $\Hil$, and by $\Theta_i\colon\IB(\Hil)\to\IB(\Hil_i)$
the corresponding compression.
 From the above, each $\Theta_i\circ\p$ has a c.p.\ extension $\bar{\p}_i\colon\cA\to\IB(\Hil_i)$.
Since $\|\bar{\p}_i\|\le\|\p(1)\|$ for all $i$, one may find a limit point $\bar{\p}$ of $\bar{\p}_i$
in the point-ultraweak topology.
It is not too hard to verify that $\bar{\p}$ is a c.p.\ extension of $\p$.
\end{proof}

\begin{thm}[Kadison--Choi Inequality and Choi's Multiplicative Domain]\label{thm:choi}
Let $\cA$ be a \spc and $\p\colon\cA\to\IB(\Hil)$ be a u.c.p.\ map.
Then, for every $x,y\in\cA$, one has
$\p(x^*x)\geq\p(x)^*\p(x)$ and
\[
\|\p(y^*x)-\p(y)^*\p(x)\|\le\|\p(y^*y)-\p(y)^*\p(y)\|^{1/2}\|\p(x^*x)-\p(x)^*\p(x)\|^{1/2}.
\]
In particular,
\[
\md(\p)=\{ x\in\cA : \p(x^*x)=\p(x)^*\p(x)\mbox{ and }\p(xx^*)=\p(x)\p(x)^*\}
\]
is a unital $*$-subalgebra of $\cA$ and $\p|_{\md(\p)}$ is a positive \shom.
\end{thm}
\begin{proof}
This follows from the Stinespring Dilation Theorem.
Since $\p$ can be expressed as $\p(x)=V^*\pi(x)V$, one has
\[
\p(x^*x)-\p(x)^*\p(x) = V^*\pi(x)^*(1-VV^*)\pi(x)V \geq0
\]
and
\[
\p(y^*x)-\p(y)^*\p(x)=((1-VV^*)^{1/2}\pi(y)V)^*((1-VV^*)^{1/2}\pi(x)V).
\]
The rest is trivial.
\end{proof}
\begin{thm}[Pisier's Linearization Trick]
Let $\cS\subset\cA$ be a semi-operator system,
$B$ be a \ca, and $\p\colon \cS\to B$ be a u.c.p.\ map.
Assume that
\[
\{ x\in\cS : 1-x^*x\in\arch(\cA_+)\mbox{ and }\p(x)\in B_\ru\}
\]
generates $\cA$ as a $*$-algebra. Then, $\p$ extends to
a positive \shom from $\cA$ into $B$.
\end{thm}
\begin{proof}
Let $B\subset\IB(\Hil)$. By Theorem~\ref{thm:arveson}, $\p$ extends to a u.c.p.\
map $\bar{\p}\colon\cA\to\IB(\Hil)$.
If $x$ is in the set described in the theorem, then by Theorem~\ref{thm:choi} one has
\[
1=\bar{\p}(x)^*\bar{\p}(x)\le\bar{\p}(x^*x)\le 1,
\]
which implies $\bar{\p}(x)^*\bar{\p}(x)=\bar{\p}(x^*x)$.
Likewise for $x^*$, and one has $x\in\md(\bar{\p})$.
Thus, $\bar{\p}$ is multiplicative on $\cA$ and it maps into $B$.
\end{proof}

\begin{cor}\label{cor:pisier}
Let $\cS\subset\cA$ be a semi-operator system which contains
enough unitary elements of $\cA$ to generate $\cA$ as a $*$-algebra.
Then, every \shom $\pi$ from $\cA$ into a
\ca $B$ is positive provided that $\pi|_\cS$ is c.p.
Moreover,
\[
\arch(\cA_+) = \arch( \{\sum_{i,j} x_i^*a_{i,j}x_j : n\in\IN,\ [a_{i,j}]_{i,j}\in\IM_n(\cS)_+,\ x_1,\ldots,x_n\in\cA\}).
\]
\end{cor}
\begin{proof}
The first assertion is immediate.
For the second, observe that the right hand side defines a
$*$-positive cone on the $*$-algebra $\cA$, which gives rise to the
same universal \ca as that of the original $\cA$.
\end{proof}
\subsection{Kirchberg's Theorem on $\mathrm{C}^*\IF_2\otimes\IB(\ell_2)$}
We will prove the following celebrated theorem of Kirchberg.
The proof takes a similar line as \cite{pisier,fp}.
\begin{thm}[\cite{kirchberg}]\label{thm:kirchberg}
Let $A_i$ be \ca{}s and $A=A_1*A_2$ be their unital full free product \ca.
If $A_i\otimes_{\max}\IB(\ell_2)=A_i\otimes_{\min}\IB(\ell_2)$ for each $i$, then
$A\otimes_{\max}\IB(\ell_2)=A\otimes_{\min}\IB(\ell_2)$.
In particular, one has
\[
\mathrm{C}^*\IF_d\otimes_{\max}\IB(\Hil)=\mathrm{C}^*\IF_d\otimes_{\min}\IB(\Hil)
\]
for every $d$ and every Hilbert space $\Hil$.
\end{thm}

The proof requires a description of the operator system structure of
$A_1+A_2\subset A$.
\begin{lem}[\cite{kavruk}]\label{lem:kavruk}
Let $\cA=\cA_1*\cA_2$ be a unital full free product and
$\cS=\cA_1+\cA_2\subset \cA$ be the semi-operator system.
Then, $a\in \cS\otimes\IB(\ell_2)$ belongs to $(\cA\otimes^{\min}\IB(\ell_2))_+$
if and only if there are $a_i\in (\cA_i\otimes^{\min}\IB(\ell_2))_+$ such that
$a=a_1+a_2$.
\end{lem}
\begin{proof}
The ``if'' direction is trivial. For the other direction,
we first deal with the case $a \in (\cS\otimes^{\min}\IB(\ell_2^n))_+$.
Let us consider the cone
\[
C=\{ a_1 + a_2 : a_i\in (\cA_i\otimes^{\min}\IB(\ell_2^n))_+ \}.
\]
Since $C$ is algebraically solid, if $a\notin\arch(C)$, then there is a
self-adjoint linear functional
$\p\colon \cS\otimes\IB(\ell_2^n)\to\ik$ such that
\[
\p(a)< 0 \le \inf_{c\in C}\p(c).
\]
Then, $\p_i=\p|_{\cA_i\otimes^{\min}\IB(\ell_2^n)}$ are self-adjoint positive linear functionals,
and by Lemma~\ref{lem:corr} they correspond to c.p.\ maps $\psi_i\colon \cA_i\to\IB(\ell_2^n)$.
Let $p$ be the support projection of $h=\psi_1(1)=\psi_2(1)$.
Then the maps $\psi'_i(\,\cdot\,)=h^{-1/2}\psi_i(\,\cdot\,)h^{-1/2}|_{p\ell_2^n}$
are u.c.p.\ and extend to a u.c.p.\ map $\psi'$ on $\cA$ by Theorem~\ref{thm:boca}.
It follows that $(h^{1/2}\otimes 1)(\psi'\otimes\id(a))(h^{1/2}\otimes1)\geq0$.
This implies $\p(a)\geq0$ in contradiction with the hypothesis.
Thus, we have shown that $(\cS\otimes^{\min}\IB(\ell_2^n))_+\subset\arch(C)$.

Now, let $a \in (\cS\otimes^{\min}\IB(\ell_2))_+$.
By the previous result, for every $n$, there are $a_i(n)\in(\cA_i\otimes^{\min}\IB(\ell_2^n))_+$
such that $(\id\otimes\Theta_n)(a+1/n)=a_1(n)+a_2(n)$, where
$\Theta_n\colon\IB(\ell_2)\to\IB(\ell_2^n)$ is the compression.
For $R>0$ such that $R1\geq a$, one has $0\le a_i(n) \le (R+1)1$ for all $n$.
Take unital finite-dimensional subspaces $E_i\subset\cA_i$ such that
$a\in (E_1+E_2)\otimes \IB(\ell_2)$.
Since $\cA_1\cap\cA_2=\ik1$, one has $a_i(n)\in E_i\otimes\IB(\ell_2^n)$.
Therefore, one may find ultraweak limit points $a_i$ of $(a_i(n))_{n=1}^\infty$
(jointly for $i=1,2$) in $E_i\otimes\IB(\ell_2)$.
It follows that $a_i\geq0$ and $a=a_1+a_2$.
\end{proof}

\begin{proof}[Proof of Theorem~\ref{thm:kirchberg}]
By Corollary~\ref{cor:pisier}, it suffices to show that
the formal identity map from $A\otimes^{\min}\IB(\ell_2)$
into $A\otimes_{\max}\IB(\ell_2)$ is c.p.\ on $\cS\otimes\IB(\ell_2)$,
where $\cS=A_1+A_2\subset A$.
Since $\IM_n(\ik)\otimes\IB(\ell_2)\cong\IB(\ell_2)$, we only have to
show that it is positive.
But if $a\in \cS\otimes\IB(\ell_2)$ is positive,
then $a=a_1+a_2$ for some $a_i\in (A_i\otimes^{\min}\IB(\ell_2))_+$
by the above lemma.
By assumption, one has $a_i\in\arch(A_i\otimes^{\max}\IB(\ell_2))_+$.
Now, it is easy to see that $a=a_1+a_2\in\arch(A\otimes^{\max}\IB(\ell_2))$.
This completes the proof.
\end{proof}
\subsection{The positive cone of the free group algebras}
Schm\"udgen has proved in 1980s (unpublished, see \cite{nt}) that $\ik[\IF_d]_+$ is
archimedean closed. It was generalized by McCullough (\cite{mccullough}) and
Bakonyi--Timotin (\cite{bt}) to the operator-valued case. Here we give a somewhat
simpler proof of this result. Our proof employs a well-known matrix completion trick,
in conjunction with the geometry of a tree.

For notational simplicity, we fix a Hilbert space $\Hil$ throughout this section and
denote $\IB=\IB(\Hil)$. In particular, $\IB=\ik$ when $\dim \Hil=1$.
We denote by $\IB[\G]$ the group algebra of $\G$ with coefficients in $\IB$.
So, we will view an element $f\in\IB[\G]$ as a function $f\colon\G\to\IB$.
One has
\[
\IB[\G]\cong\IB(\Hil)\otimes^{\max}\ik[\G]
\ \mbox{ and }\
\IB[\G]_+=\{ \sum_{i=1}^n\xi_i^**\xi_i : n\in\IN,\ \xi_i\in\IB[\G]\}.
\]
We first characterize those elements which belong to $\IB[\G]_+$.
We denote by $\rho$ the right regular representation of $\G$ on $\ell_2\G$:
$\rho(a)\delta_t=\delta_{ta^{-1}}$ for $a,t\in\G$; and
for a finite subset $E\subset\G$,
by $\rho_E(a)$ the compression of $\rho(a)$ to $\IB(\ell_2 E)$.
We identify $\IB(\ell_2 E)$ with $\IM_E(\ik)$ and
let $\cT_E\subset\IM_E(\ik)$ be the operator system consisting
of the ``Toeplitz operators,''
\begin{align*}
\cT_E
 &=\{ \rho_E(g) : g\in\ik[\G]\}\\
 &=\{ X\in\IM_E(\ik) : \exists g\in\ik[\G] \mbox{ such that }X=[g(s^{-1}t)]_{s,t\in E} \}.
\end{align*}
\begin{lem}\label{lem:hscondition}
Let $f\in \IB[\G]_\rh$ and $E\subset\G$ be a finite subset such that $\supp f\subset E^{-1}E$.
Then, $f\in\{ \sum_{i=1}^n \xi_i^**\xi_i : n\in\IN,\ \xi_i\in \IB[E]\}$ if and only if
the linear map $\p_f\colon\cT_E\to \IB$, given by $\p_f(\rho_E(a))=f(a)$, is c.p.
One may take $n\le|E|$ in the above factorization of $f$,
and if $\dim\IB=\infty$ in addition, then one may take $n=1$.
\end{lem}
\begin{proof}
If $f=\sum_i \xi_i^**\xi_i$, then for $X=[X_{s,t}]_{s,t}=\rho_E(g)\in\cT_E$, one has
\[
\p_f(X)=\sum_a g(a)f(a)= \sum_i\sum_{a,t} g(a)\xi_i^*(at^{-1})\xi_i(t)
=\sum_i\sum_{s,t} X_{s,t}\xi_i(s)^*\xi_i(t).
\]
 From the latter expression, it is not too hard to see $\p_f$ is c.p.

Conversely, if $\p_f$ is c.p., then it extends to a c.p.\ map,
still denoted by $\p_f$, on $\IM_E(\ik)$ by Theorem~\ref{thm:arveson}.
It follows that $b=[\p_f(e_{s,t})]_{s,t\in E}$ is positive in $\IM_E(\IB)$.
Develop $b^{1/2}$ as $[\xi_i(t)]_{i,t\in E}$. Then, one has
\[
(\sum_i \xi_i^**\xi_i)(a)
=\sum_i\sum_t \xi_i(ta^{-1})^*\xi_i(t)
=\sum_t b_{ta^{-1},t}
=\p_f(\rho_E(a))=f(a).
\]
We note that if $\dim\IB=\infty$, then there are isometries $S_i\in\IB$
with mutually orthogonal ranges, and $\xi=\sum_i S_i\xi_i\in\IB[E]$
satisfies $f=\xi^**\xi$.
\end{proof}
We need the following well-known matrix completion trick. 
\begin{lem}\label{lem:matrixcompletion}
Let $\Hil_0\oplus\Hil_1\oplus\Hil_2$ be a direct sum of Hilbert spaces.
Suppose
\[
\begin{bmatrix} A & X & \questionbox\\ X^* & B & Y\\ \questionbox & Y^* & C\end{bmatrix} \in \IB(\Hil_0\oplus\Hil_1\oplus\Hil_2)
\]
is a partially defined operator matrix such that its compressions on
$\Hil_0\oplus\Hil_1$ and $\Hil_1\oplus\Hil_2$ are both positive.
Then one can complete the unspecified block $\questionbox$ so that the resultant operator matrix
is positive on $\Hil_0\oplus\Hil_1\oplus\Hil_2$.
\end{lem}
\begin{proof}
We note that $\left[\begin{smallmatrix} A & X \\ X^* & B\end{smallmatrix}\right]\geq0$
if and only if $X_1:=A^{-1/2}XB^{-1/2}$ is well-defined and contractive
(see Lemma 3.1 in \cite{paulsen}).
Likewise for $Y_1:=B^{-1/2}YC^{-1/2}$.
Now, since
\[
\left[\begin{smallmatrix}
1 & X_1 & X_1Y_1\\ X_1^* & 1 & Y_1\\ Y_1^*X_1^* & Y_1^* & 1
\end{smallmatrix}\right] = \left[\begin{smallmatrix}
1 & 0 & 0\\ X_1^* & (1-X_1^*X_1)^{1/2} & 0\\ Y_1^*X_1^* & Y_1^*(1-X_1^*X_1)^{1/2} & (1-Y_1^*Y_1)^{1/2}
\end{smallmatrix}\right] \left[\begin{smallmatrix}
1 & X_1 & X_1Y_1\\ 0 & (1-X_1^*X_1)^{1/2} & (1-X_1^*X_1)^{1/2}Y_1\\ 0 & 0 & (1-Y_1^*Y_1)^{1/2}
\end{smallmatrix}\right],
\]
the operator matrix
\[
\left[\begin{smallmatrix}
A & X & Z \\ X^* & B & Y\\ Z^* & Y^* & C
\end{smallmatrix}\right] = \left[\begin{smallmatrix}
A^{1/2} & 0 & 0\\ 0 & B^{1/2} & 0\\ 0 & 0 & C^{1/2}
\end{smallmatrix}\right] \left[\begin{smallmatrix}
1 & X_1 & X_1Y_1\\ X_1^* & 1 & Y_1\\ Y_1^*X_1^* & Y_1^* & 1
\end{smallmatrix}\right] \left[\begin{smallmatrix}
A^{1/2} & 0 & 0\\ 0 & B^{1/2} & 0\\ 0 & 0 & C^{1/2}
\end{smallmatrix}\right]
\]
is positive for $Z=XB^{-1}Y$.
\end{proof}

Let $s_1,\ldots,s_d$ be the canonical generators of the free group $\IF_d$.
A pair $\{s,t\}$ in $\IF_d$ is adjacent in the \emph{Cayley graph} of $\IF_d$ if
$st^{-1}\in\{ s_1^{\pm1},\ldots,s_d^{\pm1}\}$.
The Cayley graph of $\IF_d$ is a simplicial tree.
We say a subset $E\subset\IF_d$ is \emph{grounded}
if it contains the unit $1$ and is connected in the Cayley graph.
Thus, $E$ is grounded if and only if
$t\in E$ and $t=t't''$ without cancelation imply $t''\in E$.
\begin{lem}[Proposition 4.4 in \cite{bt}]\label{lem:equivextension}
Let $g\in\IB[\IF_d]_\rh$ and $E\subset\IF_d$ be a finite grounded subset.
If $(\rho_E\otimes\id_{\IB})(g)\geq0$ in $\IM_E(\IB)$, then there is a positive type function $g'\colon\IF_d\to\IB$
such that $g'=g$ on $E^{-1}E$.
\end{lem}
\begin{proof}
By inductive construction, it suffices to show that
for every grounded subset $E'\supset E$ such that $|E'\setminus E|=1$,
one can find $g'\in\IB[\IF_d]_\rh$ such that $g'=g$ on $E^{-1}E$ and $(\rho_{E'}\otimes\id_{\IB})(g')\geq0$.
Thus, let $E'=E\cup\{ t_0\}$ with $t_0=s_it_0''$ for $t_0''\in E$
and one of the generators $s_i$. (The case for $t_0=s_i^{-1}t_0''$ is similar.)
We claim that
\[
\{ s\in E : s^{-1}t_0 \in E^{-1}E\} = \{ s\in E : s_i^{-1}s\in E\}.
\]
Indeed, if $s=s_is''$ for some $s''\in E$, then $s^{-1}t_0=(s'')^{-1}t_0''\in E^{-1}E$.
To prove the converse inclusion, suppose $s^{-1}t_0\in E^{-1}E$.
By groundedness, there is a non-trivial decomposition $t_0=pq$ without cancelation
such that $p^{-1}s\in E$. Note that the first letter of $p$ is $s_i$.
If the reduced form of $s$ starts by $s_i$, then $s_i^{-1}s\in E$ by groundedness.
Otherwise, there is no cancelation between $p^{-1}$ and $s$, and hence $p^{-1}s\in E$
implies $s_i^{-1}s\in E$ by groundedness. The claim is proved.
We denote the above subset by $E_1$ and set $E_0=E\setminus E_1$.

Now we extend $\rho_E(g)$ to a partially defined matrix $X=[X_{s,t}]_{s,t\in E'}$
such that $X_{s,t}=g(s^{-1}t)$ for $(s,t)\in E'\times E'$ with $s^{-1}t\in E^{-1}E$.
The unspecified entries of $X$
are those for $(s,t_0)$ and $(t_0,s)$ with $s\in E_0$.
We apply Lemma~\ref{lem:matrixcompletion} to
$\Hil_0=\ell_2 E_0$, $\Hil_1=\ell_2 E_1$, and $\Hil_2=\ell_2\{t_0\}$.
We note that the compression of $X$ on $\Hil_1\oplus\Hil_2$ is positive, thanks to
the right equivariant map $E_1\cup\{ t_0 \}\ni s\mapsto s_i^{-1}s\in E$.
Thus, one obtains a fully defined positive matrix $X$.
We observe that $s^{-1}t_0\neq t_0^{-1}s'$ for any $s,s'\in E_0$.
Indeed, the shortest right segments that does not belong to $E$ is
$s_it_0''$ in the left hand side and $s_i^{-1}s'$ in the right hand side.
It follows that the entry $X_{s,t}$ of the positive matrix $X\in\IM_{E'}(\IB)$ depends
only on $s^{-1}t$ and $X$ is of the form $(\rho_{E'}\otimes\id_{\IB})(g')$.
\end{proof}

The following is slightly more precise than
Theorem 7.1 in \cite{bt} (and Theorem 0.1 in \cite{mccullough}).
It implies $\IB(\Hil)\otimes^{\max}\ik[\IF_d]=\IB(\Hil)\otimes^{\min}\ik[\IF_d]$,
and hence Theorem~\ref{thm:kirchberg}.

\begin{thm}[\cite{bt}]\label{thm:schmudgen}
Let $\IB=\IB(\Hil)$, $f\in \IB[\IF_d]_\rh$, and $E\subset\IF_d$ be a grounded subset
such that $\supp f\subset E^{-1}E$.
Assume that $(\id_{\IB}\otimes\pi)(f)\geq0$ for
every finite-dimensional unitary representation $\pi$ of $\IF_d$
(of dimension at most $2|E|\dim\Hil$).
Then, there are $n\le|E|$ and $\xi_1,\ldots,\xi_n\in \IB[\IF_d]$ such that
$\supp\xi_i\subset E$ and $f=\sum_{i=1}^n \xi_i^**\xi_i$.
If $\dim\IB=\infty$ in addition, then one may take $n=1$.
\end{thm}
\begin{proof}
By Lemma~\ref{lem:hscondition} (we may assume $E$ is finite),
it suffices to show $\p_f\colon\cT_E\to\IB(\Hil)$ is c.p.
Let $g\colon\IF_d\to\IM_k(\ik)$ be such that $(\rho_E\otimes\id)(g)\geq0$.
By Lemma~\ref{lem:equivextension}, we may assume that $g$ is of positive type.
Hence, by Corollary~\ref{cor:posdeffunct}, there are a unitary representation
$\sigma$ on $\Hil_\sigma$ and $V\in\IB(\ell_2^k,\Hil_\sigma)$ such that
$g(a)=V^*\sigma(a)V$.
One has to show that
\[
(\p_f\otimes\id)((\rho_E\otimes\id)(g))=\sum_{a\in E^{-1}E} f(a)\otimes g(a)
=(1\otimes V)^*((\id_{\IB}\otimes\sigma)(f))(1\otimes V)
\]
is positive. We will follow Choi's proof of RFD for $\mathrm{C}^*\IF_d$.
Suppose the contrary that $(\id_{\IB}\otimes\sigma)(f)$ is not positive.
Then, there is a vector $\zeta$ in the algebraic tensor product $\Hil\otimes\Hil_\sigma$
such that $\ip{(\id_{\IB}\otimes\sigma)(f)\zeta,\zeta}<0$.
We may find a finite-dimensional subspace $\Hil_0\subset\Hil_\sigma$
of dimension at most $\dim\Hil$ such that $\zeta\in \Hil\otimes\Hil_0$.
Let $\Hil_1=\lh\sigma(E)\Hil_0$ and
write $X_i=P_{\Hil_1}\sigma(s_i)|_{\Hil_1}\in\IB(\Hil_1)$
for the generators $s_1,\ldots,s_d\in\IF_d$. We define the
unitary representation $\pi$ of $\IF_d$ by assigning to each $s_i$ a unitary element
\[
\pi(s_i)=\begin{bmatrix} X_i & (1-X_iX_i^*)^{1/2}\\ (1-X_i^*X_i)^{1/2} & -X_i^*\end{bmatrix}
 \in\IM_2(\IB(\Hil_1))_\ru.
\]
It follows that
$(1\otimes\pi)(t)(\zeta\oplus 0)=(1\otimes\sigma)(t)\zeta\oplus 0$ for every $t\in E$,
and
$\ip{(\id_{\IB}\otimes\pi)(f)(\zeta\oplus 0),\zeta\oplus 0}=\ip{(\id_{\IB}\otimes\sigma)(f)\zeta,\zeta}<0$,
in contradiction with the assumption.
\end{proof}

Rudin (\cite{rudin}) has proved existence of $f\in\arch(\ik[\IZ^2]_+)$
such that $\supp f\subset E^{-1}E$ with $E=[0,N]^2$ but $f$ \emph{cannot} be expressed as
$\sum_i \xi_i^**\xi_i$ for any $\xi_i\in\ik[E]$. This implies that the matrix completion
problem of the following type has a negative answer in general:
Given Toeplitz matrices $A_k=A_{-k}^*\in\IM_n(\ik)$, $|k|\le m-1$, whose
Toeplitz operator matrix $[A_{i-j}]_{i,j=1}^{m}\in\IM_{m}(\IM_n(\ik))$
is positive, can one find a Toeplitz matrix $A_m=A_{-m}^*\in\IM_n(\ik)$ so that
$[A_{i-j}]_{i,j=1}^{m+1}\in\IM_{m+1}(\IM_n(\ik))$ is still positive?
We note that nonetheless $\ik[\IZ^2]_+$ is archimedean closed (\cite{scheiderer}).
\subsection{Kirchberg's Conjecture}\label{sec:kirchberg}
In his seminal paper \cite{kirchberg}, Kirchberg has shown that
the Connes Embedding Conjecture (see Section~\ref{sec:cec}) is
equivalent to several other important conjectures in operator algebra theory,
one of which is Kirchberg's Conjecture (KC) that
\[
\mathrm{C^*}\IF_d\otimes_{\max}\mathrm{C^*}\IF_d
=\mathrm{C^*}\IF_d\otimes_{\min}\mathrm{C^*}\IF_d
\]
holds for some/all $d\geq2$.
In this section, we assume the scalar field is complex, $\ik=\IC$,
and give the proof of the equivalence. The proofs of this section are
analytically involved, not so self-contained, and probably off the
scope of this note, but we include them because some results seem to
be new or at least not well documented in the literature.
Consult \cite{bo} for technical terms in the proofs which are not
explained in this note.

Recall that the \emph{opposite algebra} of an algebra $\cA$ is the algebra
\[
\cA^{\op}=\{ a^{\op} : a\in\cA\}
\]
which has the same linear and $*$-positive structures as $\cA$,
but has opposite multiplication $a^{\op}b^{\op}=(ba)^{\op}$.
It is $*$-isomorphic to the \emph{complex conjugate} $\bar{\cA}$
of $\cA$ by $a^{\op}\leftrightarrow \bar{a}^*$.
We note that $\IC[\G]^\op\cong\IC[\G]$ via $f\leftrightarrow \check{f}$, where $\check{f}(s)=f(s^{-1})$.

\begin{thm}[\cite{kirchberg}]\label{thm:kuq}
Let $A$ be a \ca with a tracial state $\tau$, and $u_1,u_2,\ldots$ be
a dense sequence in the unitary group $A_\ru$ of $A$.
Let $(M,\tau)$ be a finite von Neumann algebra and suppose that
there is a sequence $v_1,v_2,\ldots$ of unitary elements in $M$
such that $\tau(u_i^*u_j)=\tau(v_i^*v_j)$ for all $i,j$.
Then, there is a projection $p\in M$ such that $\pi_\tau(A)''$
is embeddable into $pMp\oplus(p^\perp M p^\perp)^{\op}$.
In particular, if $M$ moreover satisfies $p^\perp Mp^\perp\cong(p^\perp M p^\perp)^{\op}$
(e.g., $M=R^\omega$), then $\pi_\tau(A)''\hookrightarrow M$.
\end{thm}
\begin{proof}
We reproduce Kirchberg's proof (\cite{kirchberg}) here.
The map $u_i\mapsto v_i$ extends to a linear isometry $\p$
from $L^2(A,\tau)$ into $L^2(M,\tau)$.
Since the set of unitary elements is closed in $L^2(M,\tau)$,
the map $\p$ sends unitary elements to unitary elements by density and continuity.
By the Russo--Dye theorem, $\p$ is a contraction from $A$ into $M$. Thus the unital map
$\theta$ defined by $\theta(a)=\p(1)^*\p(a)$ is a positive contraction from
$A$ into $M$, which sends unitary elements to unitary elements and preserves the tracial state.
By a multiplicative domain argument (Corollary 2.6 in \cite{qwep}),
the map $\theta$ is a Jordan morphism. Thus, there is a central
projection $p$ in $\theta(A)''$ such that $p\theta$ (resp.\ $p^\perp\theta$)
is a \shom (resp.\ an opposite \shom).
\end{proof}

Let $\tau$ be a tracial state on a \spc $\cA$.
Then, besides the GNS representation $\pi_\tau$ of $\cA$ on $L^2(\cA,\tau)$,
there is a representation $\pi_\tau^\op$ of $\cA^{\op}$ on $L^2(\cA,\tau)$,
given by
$\pi_\tau^\op(b^\op)\hat{x}=\widehat{xb}$. Since $\tau$ is tracial, one has
\[
\ip{\pi_\tau^\op(b^\op)\hat{x},\hat{x}}=\tau(x^*xb)=\ip{\pi_\tau(b)\hat{x}^*,\hat{x}^*}\geq0
\]
if $b^\op\geq0$, which means that $\pi_\tau^\op$
is a positive $*$-representation on $L^2(\cA,\tau)$.

Since $\pi_\tau(\cA)$ and $\pi_\tau^\op(\cA^\op)$ commute, one obtains a positive \srep
\[
\pi_\tau\times\pi_\tau^\op\colon\cA\otimes\cA^\op\to\IB(L^2(\cA,\tau)),\
(\pi_\tau\times\pi_\tau^\op)(a\otimes b^\op)\hat{x}=\widehat{axb}.
\]
In particular,
\[
\sum a_i\otimes b_i^{\op} \mapsto
\ip{(\pi_\tau\times\pi_\tau^\op)(\sum a_i\otimes b_i^\op)\hat{1},\hat{1}} = \tau(\sum a_i b_i )
\]
is a state on $\cA\otimes^{\max}\cA^{\op}$.

To prove the equivalence between CEC and KC, Kirchberg (\cite{kirchberg}) has shown
that all von Neumann algebras are QWEP if it is the case for finite von Neumann algebras.
We take a different route and prove the following alternative.

\begin{thm}\label{thm:kirchbergm}
There is a tracial state $\tau$ on $\mathrm{C^*}\IF_d$ such that
$\pi_\tau\times\pi_\tau^\op$ is faithful on the \ca
$\mathrm{C^*}\IF_d\otimes_{\max}\mathrm{C^*}\IF_d$.
This means that if $\tau(\sum_i a_ixb_ix^*)\geq0$ for every
$x\in\IC[\IF_d]$ and $\tau\in T(\IC[\IF_d])$,
then $\sum_i a_i\otimes \check{b}_i\in\arch(\IC[\IF_d\times\IF_d]_+)$.
\end{thm}
\begin{proof}
First, take a faithful \srep $\pi\times\pi'$ of
$\mathrm{C^*}\IF_d\otimes_{\max}(\mathrm{C^*}\IF_d)^\op$ on $\Hil$.
By inflating it if necessary, one may find a von Neumann algebra $M$
in its standard form (see Section IX.1 in \cite{takesakiII})
such that $\pi(\mathrm{C^*}\IF_d)\subset M$ and $\pi'((\mathrm{C^*}\IF_d)^\op)\subset M'\cong M^\op$.
Further, by replacing $\pi$ with
$\pi\oplus(\pi')^\op\colon \mathrm{C^*}\IF_d\to M\oplus M\subset\IM_2(M)$,
and $M$ with $\IM_2(M)$, one may assume $\pi'=\pi^{\op}$.
There is an $\IR$-action on $M$ such that the crossed product
von Neumann algebra $M\rtimes\IR$ is semi-finite (Theorem XII.1.1 in \cite{takesakiII}).
We think $M\rtimes\IR$ acts on $\Hil\otimes L^2(\IR)$ by
$a\otimes 1$ for $a\in M$ and $u_t\otimes\lambda_t$ for $t\in\IR$, where $u_t$
is the implementing unitary group.
Hence $\mathrm{C^*}\IF_d\hookrightarrow M\rtimes\IR$ is given by $\tilde{\pi}(a)=\pi(a)\otimes1$.
Then under the identification of $\Hil\otimes L^2(\IR)$ with $L^2(\IR,\Hil)$,
the embedding $\tilde{\pi}^\op\colon (\mathrm{C^*}\IF_d)^\op\hookrightarrow (M\rtimes\IR)'$ is given by
$(\tilde{\pi}^{\op}(b^{\op})\xi)(t)=u_t\pi^{\op}(b^\op)u_t^*\xi(t)$.
Since $\pi\times\pi^\op$ is weakly contained in $\tilde{\pi}\times\tilde{\pi}^\op$,
one may assume from the beginning that $M$ is a semi-finite von Neumann algebra with
a faithful normal tracial weight $\Tr$, and the \srep $\pi\times\pi^\op$ of
$\mathrm{C^*}\IF_d\otimes_{\max}(\mathrm{C^*}\IF_d)^\op$ on $L^2(M,\Tr)$ is faithful.
Since vectors with finite supports are dense in $L^2(M,\Tr)$, one can use Choi's trick
(see the proof of Theorem~\ref{thm:schmudgen}) and find a countable family of
\srep{}s $\sigma_k$ from $\mathrm{C^*}\IF_d$ into finite von Neumann algebras
$(M_k,\tau_k)$ such that $\bigoplus\sigma_k\times\sigma^\op_k$ is faithful.
By considering $\tau=\sum2^{-k}\tau_k\in T(\mathrm{C^*}\IF_d)$, we are done.
\end{proof}

We denote by $\IF_\infty=\langle s_1,s_2,\ldots\rangle$ the free group on countably many generators.
We identify the elements $s_i$ with the corresponding unitary elements in $\mathrm{C^*}\IF_\infty$.
\begin{thm}\label{thm:cecmin}
Let $A$ be a \ca having a tracial state $\tau$ and a dense sequence $u_1,u_2,\ldots$ of
unitary elements in $A$.
Let $\sigma\colon \mathrm{C^*}\IF_\infty\to A$, $\sigma(s_i)=u_i$ be the corresponding \shom.
Then, the finite von Neumann algebra $\pi_\tau(A)''$ satisfies CEC if and only if
the linear functional defined by
\[
s_i\otimes s_j\mapsto\tau(u_iu_j^*)
\ \mbox{ on }\
\lh\{ s_i\otimes s_j\}\subset \mathrm{C^*}\IF_\infty\otimes_{\min}\mathrm{C^*}\IF_\infty
\]
is contractive.
\end{thm}

In the operator space terminology, $\lh\{s_1,s_2,\ldots\}$ is denoted by $\ell_1$, and
the above theorem in particular says that KC holds true if
the formal identity $\ell_1\otimes_{\min}\ell_1\hookrightarrow
 \mathrm{C^*}\IF_\infty\otimes_{\max}\mathrm{C^*}\IF_\infty$
is contractive. It is known that this map is bounded (by $K^{\IC}_G<1.41$).
In case this map is moreover completely contractive, the theorem easily
follows from Pisier's trick. See Section 12 in \cite{pisier:g} for more information.

\begin{proof}
We first prove the `if' part. For notational simplicity,
denote $C=\mathrm{C^*}\IF_\infty$ and let it act on a Hilbert space $\Hil$
with infinite multiplicity.
We note that $C$ is $*$-isomorphic to its complex conjugate $\bar{C}$ acting
on the conjugate Hilbert space $\bar{\Hil}$ by $\bar{s}$, $s\in\IF_\infty$.
We use the standard identification of the Hilbert space $\Hil\otimes\bar{\Hil}$
with the space $S_2(\Hil)$ of Hilbert--Schmidt class operators on $\Hil$.
Then, $a\otimes\bar{b}\in C\otimes\bar{C}$ acts on $S_2(\Hil)$ by $(a\otimes\bar{b})x=axb^*$.
By considering the complete isometries
$s_i\mapsto s_1^{-1}s_i$, one may assume that $s_1=1$
and $\IF_\infty=\langle s_2,s_3,\ldots\rangle$ in the above formula.
Then, since every unital contraction is positive on a $\mathrm{C}^*$-algebra,
the linear functional $s_i\otimes\bar{s}_j\mapsto \tau(u_iu_j^*)$ extends to
a state $f$ on $C\otimes_{\min}\bar{C}$.
Let $n\in\IN$ be given and $\e=1/n$.
Then, there is $x\in S_2(\Hil)$ of norm $1$ such that
\[
|f(s_i\otimes\bar{s}_j) - \Tr(s_ixs_j^*x^*)|<\e^2/8
\]
for $i,j\le n$. Since $\|s_ix - xs_i \|_2^2=2(1-\Re\Tr(s_ixs_i^*x^*))<(\e/2)^2$,
denoting $h=x^*x$, one has
\[
|\tau(u_iu_j^*) - \Tr(h s_is_j^*)|<\e
\ \mbox{ and }\
\| s_ihs_i^* - h \|_1<\e
\]
for $i,j\le n$. Here $\|\,\cdot\,\|_1$ denotes the trace norm.
By approximation, one may assume that $h$ is of finite-rank and has rational eigenvalues.
Then, Lemma 6.2.5 in \cite{bo} implies that there is a u.c.p.\ map
$\p_n$ from $C$ into the hyperfinite $\mathrm{II}_1$-factor $R$
such that $\tau(\p_n(a))=\Tr(ha)$ for $a\in C$ and
$|1 - \tau(\p_n(s_i)\p_n(s_i)^*)|\le 2\e^{1/2}$ for $i\le n$.
See also Lemma 6.2.6 in \cite{bo}.
It follows that, for each $i$, the sequence $(\p_n(s_i))_{n=1}^\infty$ gives rise to
a unitary element $v_i$ in $R^\omega$ such that
\[
\tau(u_iu_j^*) = \tau(v_iv_j^*)
\]
for all $i,j$. By Theorem~\ref{thm:kuq}, $\pi_\tau(A)''\hookrightarrow R^\omega$.

For the converse implication, let $\pi_\tau(A)''\hookrightarrow R^\omega$
and $(u_i(n))_{n=1}^\infty\in\prod R$ be liftings of $u_i$ in $R^\omega$.
Then, for the corresponding \shom $\sigma_n\colon C\to R$, $s_i\mapsto u_i(n)$,
the \shom $\sigma_n\times\sigma_n^{\op}$ from $C\otimes C^{\op}$ to $\IB(L^2(R,\tau))$
is continuous with respect to the minimal tensor product.
Hence the map
\[
s_i\otimes s_j\mapsto\lim_n\ip{(\sigma_n\times\sigma_n^{\op})(s_i\otimes s_j)\hat{1},\hat{1}}
=\lim_n\tau(u_i(n)u_j(n)^*)=\tau(u_iu_j^*)
\]
is contractive with respect to the minimal tensor norm.
\end{proof}
\begin{cor}[\cite{kirchberg}]
CEC is equivalent to KC that
$\mathrm{C^*}\IF_d\otimes_{\max}\mathrm{C^*}\IF_d
=\mathrm{C^*}\IF_d\otimes_{\min}\mathrm{C^*}\IF_d$
holds for some/all $d\geq2$.
\end{cor}
\begin{proof}
Since the free groups $\IF_d$'s are embedded into each other (for $d\geq2$),
KC does not depends on $d\geq2$.
By Theorem~\ref{thm:kirchbergm}, there is an embedding $\pi$
of $\mathrm{C^*}\IF_d$ into a finite von Neumann algebra $M$ with separable predual
(which may be assumed to be a factor by considering a free product)
such that $\pi\times\pi^\op$ is faithful on $\mathrm{C^*}\IF_d\otimes_{\max}\mathrm{C^*}\IF_d$.
If $M$ is embeddable into $R^\omega$, then one can lift the \srep $\pi$ into $R^\omega$ to
$(\pi_n)_{n=1}^\infty$ into $\prod R$ and sees that $\bigoplus\pi_n\times\pi_n^\op$
is faithful. Since $R$ is hyperfinite, $\pi_n\times\pi_n^\op$ factors through
$\mathrm{C^*}\IF_d\otimes_{\min}\mathrm{C^*}\IF_d$ and KC follows.
The converse implication follows from Theorem~\ref{thm:cecmin}.
\end{proof}
We remark that the free group $\IF_d$ in KC
can be replaced with any other single
group which contains $\IF_2$ and whose full group \ca has the so-called local lifting property.
In particular, KC is equivalent to the same statement but
$\IF_d$ is replaced with a ``nontrivial'' free product of amenable groups.
See \cite{bo,kirchberg,qwep} for the proof of these facts and more information.
\subsection{Quasi-diagonality}\label{sec:qd}
In this section, we prove a weaker version of Kirchberg's Conjecture.
We still assume $\ik=\IC$.
Since $\mathrm{C^*}\IF_d$ is RFD, so is the minimal tensor product
$\mathrm{C^*}\IF_d\otimes_{\min}\mathrm{C^*}\IF_d$.
On the other hand, since any finite-dimensional \srep of
the maximal tensor product factors through the minimal tensor product,
KC is equivalent to the assertion that
$\mathrm{C^*}(\IF_d\times\IF_d)$ (which is canonically isomorphic
to $\mathrm{C^*}\IF_d\otimes_{\max}\mathrm{C^*}\IF_d$) is RFD.
By \pstz (Corollary~\ref{cor:poss}), one obtains the following.

\begin{cor}
KC is equivalent to the following statement.
If $f\in\ik[\IF_d\times\IF_d]_\rh$ is such that $\pi(f)\geq0$ for all
finite(-dimensional) unitary representations $\pi$ of $\IF_d\times\IF_d$,
then $f\in\arch(\ik[\IF_d\times\IF_d]_+)$.
\end{cor}

While residual finite-dimensionality of $\mathrm{C^*}(\IF_d\times\IF_d)$
(Kirchberg's Conjecture) remains open, one can prove that it satisfies
the following weaker property.

\begin{defn}
A \ca $A$ is said to be \emph{quasi-diagonal} if
there is a net of u.c.p.\ maps $\p_n\colon A\to\IM_{k(n)}(\IC)$ such that
\[
\|\p_n(xy)-\p_n(x)\p_n(y)\|\to0
\]
for every $x,y\in A$, and
\[
\{ x \in A : \p_n(x)\geq0 \mbox{ for all }n\}=A_+.
\]
\end{defn}

\begin{thm}[\cite{bo}]\label{thm:qd}
The \ca $\mathrm{C^*}(\IF_d\times\IF_d)$ is quasi-diagonal.
\end{thm}

The proof requires homotopy theory for \ca{}s.
It is a celebrated theorem of Voiculescu that quasi-diagonality
is a homotopy invariant. See Chapter 7 in \cite{bo} for the proof
of it and more information on quasi-diagonality.
Recall that two \shom{}s $\pi_0$, $\pi_1$ from $A$ to $B$
are \emph{homotopic}, denoted by $\pi_0\sim_{\rh}\pi_1$, if they
are connected by a pointwise continuous path $(\pi_t)_{t\in[0,1]}$
of \shom{}s; and \ca{}s $A$ and $B$ are \emph{homotopic} if
there are \shom{}s $\alpha\colon A\to B$ and $\beta\colon B\to A$
such that $\beta\circ\alpha\sim_{\rh}\id_A$ and
$\alpha\circ\beta\sim_{\rh}\id_B$.
We say $A$ is \emph{homotopically contained} in $B$ if
there are a \ca $B_1$ which is homotopic to $B$, a \shom{}
$\alpha\colon A\to B_1$, and a u.c.p.\ map $\beta\colon B_1\to A^{**}$
such that $\beta\circ\alpha=\iota_A$, the canonical embedding of $A$ into $A^{**}$.
We note that if there are \shom{}s $\alpha\colon A\to B$
and $\beta\colon B\to A^{**}$ such that $\beta\circ\alpha\sim_{\rh}\iota_A$,
then $A$ is homotopically contained in $B$.
Indeed, if $\gamma_t$ is a homotopy from $\iota_A$ to $\beta\circ\alpha$, then
$\gamma\oplus\alpha$ gives rise to an embedding of $A$ into the mapping cylinder
\[
Z_\beta=\{ (f,b)\in C([0,1],A^{**})\oplus B : f(1)=\beta(b)\},
\]
which is homotopic to $B$.

\begin{proof}[Proof of Theorem~\ref{thm:qd}]
This follows from Voiculescu's theorem, once we observe that
the property being homotopically contained in the trivial \ca $\IC1$
(or finite-dimensional $\mathrm{C}^*$-algebras, see \cite{act})
is preserved by unital full free products and tensor products.
Note that $\mathrm{C}^*\IZ \cong C(\IT)\subset C([0,1])$ has this property, to begin with.
\end{proof}

It is unclear for which group $\G$, the full group \ca{} $\mathrm{C}^*\G$ is quasi-diagonal
(or homotopically contained in $\IC1$---this is not the case when $\mathrm{C}^*\G$
has nontrivial projections, e.g., when $\G$ has torsion or Kazhdan's property $\mathrm{(T)}$).
A well-known conjecture of Rosenberg (see \cite{cde}) asserts that all amenable groups
should have quasi-diagonal $\mathrm{C}^*\G$.
On the other extreme, any infinite simple Kazhdan's property (T) group gives
rise to a counterexample (\cite{kirchberg:t}).
The case for $\mathrm{SL}_3(\IZ)$ is unclear (cf.\ \cite{bekka,bekkas}).
\subsection{Operator systems}\label{sec:opsys}
We will prove the Choi--Effros theorem giving the abstract characterization of semi-operator systems.
Recall that a \emph{(semi-)operator system} is a unital $*$-subspace $\cS$
of a (semi-pre-)\ca $\cA$, equipped with the matricial positive cone
$\IM_n(\cS)_+:=\IM_n(\cS)\cap\IM_n(\cA)_+$ for each $n\in\IN$.
Here, we do not assume it norm-closed.
We generally think that the specific embedding $\cS\subset\cA$ is \emph{not} a part of the
operator system structure of $\cS$, but only the matricial positive cones
$(\IM_n(\cS)_+)_{n=1}^\infty$ are.
A map $\p\colon \cS\to \cT$ between operator systems is said to be \emph{completely positive}
(c.p.) if $\id\otimes\p\colon \IM_n(\cS)\to\IM_n(\cT)$ is positive for all $n$.
Thus $\cS$ and $\cT$ are \emph{completely order isomorphic} if there is a unital c.p.\
(u.c.p.) linear isomorphism $\p\colon\cS\to\cT$ such that $\p^{-1}$ is also c.p.
It is not too hard to see that a semi-operator system $(\cS, (\IM_n(\cS)_+)_{n=1}^\infty)$
satisfies the following axiom:
\begin{enumerate}[(i)]
\item
$\IR_{\geq0}1\in\IM_n(\cS)_+$
and $\lambda a + b \in \IM_n(\cS)_+$ for every $a,b \in \IM_n(\cS)_+$ and $\lambda\in\IR_{\geq0}$;
\item
$x^*ax\in\IM_n(\cS)_+$ for every $m,n\in\IN$, $a\in\IM_m(\cS)_+$ and $x\in\IM_{m,n}(\ik)$;
\item For every $n$ and every $h\in\IM_n(\cS)_\rh$, there is $R>0$ such that $h\le R1$.
\interitemtext{If $\cS$ is an operator system, then it moreover satisfies the following:}
\item
$\IM_n(\cS)_+\cap(-\IM_n(\cS)_+)=\{0\}$ and $\IM_n(\cS)_+$ is archimedean closed;
\item\label{noghost}
$I(\cS):=\{ x\in\cS : \left[\begin{smallmatrix} 0 & x^* \\ x & 0\end{smallmatrix}\right]\in\arch(\IM_2(\cS)_+)\}=\{0\}$.
\end{enumerate}
We recall the ground assumption that the unit is always written as $1$, and
so is the unit matrix in $\IM_n(\ik)$. The definition of $x^*ax$ is the obvious one:
\[
x^*ax=[\sum_{k,l}x_{k,i}^*a_{k,l}x_{l,j}]_{i,j}
=[\sum_{k,l}x_{k,i}^*x_{l,j} a_{k,l}]_{i,j}
\mbox{ for }x=[x_{k,i}]_{k,i}
\mbox{ and }a=[a_{k,l}]_{k,l}.
\]
An \emph{abstract (semi-)operator system} is a system $(\cS, (\IM_n(\cS)_+)_{n=1}^\infty)$ satisfying
the above axiom. The condition (\ref{noghost}) follows from other conditions in case $\ik=\IC$.
Given an abstract semi-operator system $\cS$, we introduce the corresponding
\spc $\cA(\cS)$ as the universal unital $*$-algebra generated by $\cS$,
equipped with the $*$-positive cone
\[
\cA(\cS)_+=\{ x^*ax : n\in\IN,\ a\in\IM_n(\cS)_+,\ x\in\IM_{n,1}(\cA(\cS))\}.
\]
To see that the Combes axiom is satisfied, let $x\in\cS$ be given arbitrary.
Then, there is $R>0$ such that
$\left[\begin{smallmatrix} R1 & -x \\ -x^* & R1\end{smallmatrix}\right]\in\IM_2(\cS)_+$.
It follows that
\[
0\le \frac{1}{R}\begin{bmatrix} x^* & R1\end{bmatrix}
\begin{bmatrix} R1 & -x \\ -x^* & R1\end{bmatrix}\begin{bmatrix} x \\ R1\end{bmatrix}
=R^21-x^*x.
\]
This implies that $\cS\subset \cA(\cS)^{\mathrm{bdd}}$ and hence the Combes axiom follows.
The universal $\mathrm{C}^*$-algebra $\mathrm{C}^*_\ru(\cA(\cS))$ of $\cA(\cS)$
is generated by $\iota(\cS)$ and has the following universal property:
Every u.c.p.\ map $\cS$ into a $\mathrm{C}^*$-algebra $B$ extends
to a \shom on $\mathrm{C}^*_\ru(\cA(\cS))$.
We call $\mathrm{C}^*_\ru(\cA(\cS))$ the \emph{universal $\mathrm{C}^*$-algebra}
of the semi-operator system $\cS$ and denote it simply by $\mathrm{C}^*_\ru(\cS)$.

\begin{thm}[\cite{ce}]\label{thm:ce}
Let $\cS$ be a semi-operator system.
Then, one has
\[
\ker(\iota\colon \cS\to\mathrm{C}^*_\ru(\cS))=I(\cS)\
\]
and
\[
\arch(\IM_n(\cS)_+) = \IM_n(\cS)_\rh \cap \iota^{-1}(\IM_n(\mathrm{C}^*_\ru(\cS))_+)
\]
for every $n$.
In particular, an abstract operator system is completely order isomorphic to
a concrete operator system.
\end{thm}

\begin{proof}
By definition, one has $\IM_n(\cS)_+\subset\IM_n(\mathrm{C}^*_\ru(\cS))_+$.
Also, it is not difficult to see $I(\cS)\subset\ker\iota$.
For the converse inclusions, it suffices to show that
for every $a\in\IM_n(\cS)_\rh\setminus\arch(\IM_n(\cS)_+)$,
there is a positive \srep $\pi$ on $\mathrm{C}^*_\ru(\cS)$ such that
$(\id\otimes\pi)(a)\not\geq0$. (This in particular shows that $\ker\iota\subset I(\cS)$.)
For this, it is enough to find a u.c.p.\ map $\psi\colon\cS\to\IB(\Hil)$ such that
$(\id\otimes\psi)(a)\not\geq0$.
Since $\IM_n(\cS)_+$ is algebraically solid in $\IM_n(\cS)_\rh$,
one may find a state $\tilde{\p}$ on $\IM_n(\cS)$ such that
$\tilde{\p}(a)<0$. Let $\p\colon\cS\to\IM_n(\ik)$ be
the corresponding c.p.\ map to $\tilde{\p}$ (see Lemma~\ref{lem:corr}).
Then, $\p(1)\in\IM_n(\ik)$ is a self-adjoint positive element and
for the support projection $p$ of $\p(1)$, one has $\p(x)=p\p(x)p$ for all $x\in\cS$
(because $\left[\begin{smallmatrix} R\p(1) & \p(x)^* \\ \p(x) & R\p(1)\end{smallmatrix}\right]\geq0$
for some $R>0$).
It follows that the map $\psi$ defined by $\psi(x)=\p(1)^{-1/2}\p(x)\p(1)^{-1/2}|_{p\ell_2^n}$
is a u.c.p.\ map from $\cS$ into $\IB(p\ell_2^n)$ such that $\psi(a)\not\geq0$.
This completes the proof.
\end{proof}

The proof above is short and classical, but it would be nicer if there is another proof which
relies more directly on the definition of $\cA(\cS)_+$. There is another very
interesting \ca associated with an operator system $\cS$, called the
$\mathrm{C}^*$-envelope $\mathrm{C}^*_\re(\cS)$ of $\cS$ (see Chapter 15 in \cite{paulsen}),
but the explicit description of the corresponding \spc structure on $\cA(\cS)$ is unclear.
\subsection{Operator system duality and tensor products}
An important corollary to Theorem~\ref{thm:ce} is the operator system duality.
Let $\cS$ be a semi-operator system. We denote by $\cS^\rd$ the linear dual of $\cS$,
and identify $\IM_n(\cS^\rd)$ with the space of
the linear maps from $\cS$ to $\IM_n(\ik)$ (cf.\ Lemma~\ref{lem:corr}) to
introduce the corresponding matricial positive cones on it.
Thus for $f=[f_{i,j}]_{i,j}\in\IM_n(\cS^\rd)$, one has
$f^*=[f^*_{j,i}]_{i,j}$ (recall that $f^*(x)=f(x^*)^*$),
and $f\geq0$ in $\IM_n(\cS^\rd)$ if and only if $x\mapsto [f_{i,j}(x)]\in\IM_n(\ik)$ is
c.p. Then, the system $(\cS^\rd, (\IM_n(\cS^\rd)_+)_{n=1}^\infty)$ satisfies all the axioms
of operator systems in Section~\ref{sec:opsys}, except the ones involving the unit $1$.
We also observe by Theorem~\ref{thm:ce} (and Theorem~\ref{thm:tensorize})
that $a\in\IM_m(\cS)_\rh$ belongs to $\arch(\IM_m(\cS)_+)$ if and only if $a$ is c.p.\
as a linear map from $\cS^\rd$ to $\IM_m(\ik)$.
It follows that $\cS\hookrightarrow\cS^{\rd\rd}$ is a complete order isomorphic embedding
modulo archimedean closure.
When $\cS$ is a finite-dimensional operator system, then one can pick any faithful state $p$
on $\cS$, and make $(\cS^\rd, (\IM_n(\cS^\rd)_+)_{n=1}^\infty)$
an abstract operator system with the unit $p$. (So, the operator system structure of
$\cS^\rd$ depends on the choice of the unit $p$, and one has to choose the canonical one
if exists. For example, the unit of $\cS^{\rd\rd}=(\cS^\rd)^\rd$ has to be the unit $1$ of $\cS$.)
Summarizing the above discussion, we reach to the following corollary.
\begin{cor}[\cite{ce}]\label{cor:ce}
Let $\cS$ be a finite-dimensional operator system.
Then, its dual $\cS^\rd$ is an (abstract) operator system.
Moreover, the natural isomorphism $\cS=\cS^{\rd\rd}$ is a complete order isomorphism.
\end{cor}

Let $\cS$ and $\cT$ be (semi-)operator systems, and $\cS\otimes\cT$ be their tensor product.
There are two (in fact more) canonical operator system structures on $\cS\otimes\cT$ (\cite{kptt,fp}).
The first one is the minimal tensor product $\cS\otimes^{\min}\cT$, defined by
\[
\IM_n(\cS\otimes^{\min}\cT)_+=\IM_n(\cS\otimes\cT)
 \cap (\id\otimes\iota_\cS\otimes\iota_\cT)^{-1}
   (\IM_n(\mathrm{C}^*_\ru(\cS)\otimes^{\min}\mathrm{C}^*_\ru(\cT))_+).
\]
We observe from Theorem~\ref{thm:tensorize} that
an element $c\in\IM_n(\cS\otimes\cT)_\rh$ is positive if and only if
$(\id\otimes\p\otimes\psi)(c)\in \IM_{nkl}(\ik)_+$ for all $k,l\in\IN$ and
c.p.\ maps $\p\colon \cS\to\IM_k(\ik)$ and $\psi\colon\cT\to\IM_l(\ik)$.
The second one is the maximal operator system tensor product $\cS\otimes^{\max}\cT$,
defined by
\[
\IM_n(\cS\otimes^{\max}\cT)_+ = \arch(\{x^* (s\otimes t) x :
 x\in\IM_{kl,n}(\ik),\, s\in\IM_k(\cS)_+,\, t\in\IM_l(\cT)_+\}),
\]
where $x^* (s\otimes t) x
=[\sum_{p,p',q,q'} x_{(p,q),i}^*(s_{p,p'}\otimes t_{q,q'})x_{(p',q'),j}]_{i,j} \in \IM_n(\cS\otimes\cT)$.
We note that for $\cS_1\subset\cS_2$ and $\cT_1\subset\cT_2$, the natural
embedding $\cS_1\otimes\cT_1\hookrightarrow \cS_2\otimes\cT_2$
is a complete order isomorphic embedding with respect to minimal tensor products,
but need not be so with respect to the maximal tensor products.
It is not too hard to see that for any operator systems $\cS$ and $\cT$,
both minimal and maximal tensor product satisfy the axiom of operator systems,
and that the formal identity map from $\cS\otimes^{\max}\cT$ to $\cS\otimes^{\min}\cT$ is c.p.
These two tensor products are dual to each other in the following sense.
\begin{thm}[\cite{fp}]\label{thm:fp}
Let $\cS$ and $\cT$ be finite-dimensional operator systems.
Then the natural isomorphism
\[
(\cS\otimes^{\min}\cT)^\rd=\cS^\rd\otimes^{\max}\cT^\rd
\]
is a complete order isomorphism.
\end{thm}
\begin{proof}
First, we observe that if $p$ and $q$ are faithful states,
then $p\otimes q$ is faithful on $\cS\otimes^{\min}\cT$.
We will prove this for general operator systems.
Let $(p\otimes q)(c)=0$ for some $c\in(\cS\otimes^{\min}\cT)_+$.
Then, the faithfulness of $p$ implies $(\id\times q)(c)=0$ in $\cS$.
Thus, for every c.p.\ map $\p\colon \cS\to\IM_m(\ik)$, one has
$(\id\otimes q)((\p\otimes\id)(c))=\p((\id\times q)(c))=0$.
Since $\id\otimes q$ is also faithful on $\IM_m(\cT)$, this implies
$(\p\otimes\id)(c)=0$, and $c=0$.

By replacing $\cS$ and $\cT$ with their duals,
it suffices to show that the natural
map $\cS^\rd\otimes^{\min}\cT^\rd\to(\cS\otimes^{\max}\cT)^\rd$
is a complete order isomorphism.
For complete positivity, one has to show for any $f\in\IM_n(\cS^\rd\otimes^{\min}\cT^\rd)_+$,
the corresponding map $\tilde{f}\colon\cS\otimes^{\max}\cT\to\IM_n(\ik)$ is c.p.
But this is true, since for every $m,k,l$ and $x\in\IM_{kl,m}(\ik)$, $s\in\IM_k(\cS)_+$,
and $t\in\IM_l(\cT)_+$, one has
\[
\tilde{f}(x^* (s\otimes t) x) = x^* \ip{f,s\otimes t} x \in \IM_{mn}(\ik)_+,
\]
thanks to the completely order isomorphic embeddings
$\cS\subset\cS^{\rd\rd}$ and $\cT\subset\cT^{\rd\rd}$.
It remains to show that if $c\colon(\cS\otimes^{\max}\cT)^\rd\to\IM_n(\ik)$ is c.p.,
then it is c.p.\ on $\cS^\rd\otimes^{\min}\cT^\rd$. But by definition, $c$ is
identified with an element in
$\IM_n((\cS\otimes^{\max}\cT)^{\rd\rd})_+=\IM_n(\cS\otimes^{\max}\cT)_+$ and
$c+\e1=x^*(s\otimes t)x$ for some $x,s,t$.
Hence, for every $f\in \IM_m(\cS^\rd\otimes^{\min}\cT^\rd)_+$, one has
\[
(\id\otimes (c+\e1))(f)=x^*\ip{s\otimes t,f}x \in \IM_{nm}(\ik)_+.
\]
Since $\e>0$ was arbitrary, one has $(\id\otimes c)(f)\in \IM_{nm}(\ik)_+$.
This completes the proof.
\end{proof}
\subsection{Quantum correlation matrices and Tsirelson's Problem}
In this section, we take $\ik=\IC$.
In quantum information theory, a measurement of a state is carried out by
so-called POVMs (positive operator valued measures).
For simplicity of the presentation, we only consider PVMs (projection valued measures).
A PVM with $m$ outputs is an $m$-tuple $(P_i)_{i=1}^m$ of
orthogonal projections on a Hilbert space $\Hil$ such that $\sum P_i=1$.
For each $n\in\IN$, the convex sets of quantum correlation matrices
of two separated systems of $d$ PVMs with $m$-outputs are defined by
\[
\cQ_c^n=\{ \mb{V^*P_i^{(k)} Q_j^{(l)}V}_{\begin{subarray}{c}k,l\\ i,j\end{subarray}} :
\begin{array}{c}
\mbox{$\Hil$ a Hilbert space, $V\colon\ell_2^n\to\Hil$ an isometry}\\
\mbox{$(P_i^{(k)})_{i=1}^m$, $k=1,\ldots,d$, PVMs on $\Hil$,}\\
\mbox{$(Q_j^{(l)})_{j=1}^m$, $l=1,\ldots,d$, PVMs on $\Hil$,}\\
\mbox{$[P_i^{(k)},Q_j^{(l)}]=0$ for all $i,j$ and $k,l$}
\end{array}
\}
\]
(here $[A,B]=AB-BA$ is the commutator) and
\[
\cQ_s^n=\mathrm{closure}\{ \mb{V^*(P_i^{(k)}\otimes Q_j^{(l)})V}_{\begin{subarray}{c}k,l\\ i,j\end{subarray}} :
\begin{array}{c}
\mbox{$\Hil,\Hilk$ Hilbert spaces,}\\
\mbox{$V\colon\ell_2^n\to\Hil\otimes\Hilk$ an isometry,}\\
\mbox{$(P_i^{(k)})_{i=1}^m$, $k=1,\ldots,d$, PVMs on $\Hil$,}\\
\mbox{$(Q_j^{(l)})_{j=1}^m$, $l=1,\ldots,d$, PVMs on $\Hilk$}
\end{array}
\}.
\]
The sets $\cQ_c^n$ and $\cQ_s^n$ are closed convex subsets of $\IM_{md}(\IM_n(\IC)_+)$
such that $\cQ_s^n\subset\cQ_c^n$.
We simply write $\cQ_c$ and $\cQ_s$ when $n=1$.
Whether $\cQ_c=\cQ_s$ for every $m$ and $d$ is the well-known Tsirelson Problem.
We refer the reader to \cite{fritz,jungeetal,tsirelson} for
the background and the literature on this problem.
(The ``real'' problem may be that which of $\cQ_c$ and $\cQ_s$ should we take as
the definition of quantum correlation matrices? See \cite{fritz,act}.)
Since a commuting system on a finite-dimensional Hilbert space splits
as a tensor product if irreducible, one also has
\[
\cQ_s^n=\mathrm{closure}\{ \mb{V^*P_i^{(k)} Q_j^{(l)}V}_{\begin{subarray}{c}k,l\\ i,j\end{subarray}} :
\begin{array}{c}
\mbox{$\dim\Hil<\infty$, $V\colon\ell_2^n\to\Hil$ an isometry}\\
\mbox{$(P_i^{(k)})_{i=1}^m$, $k=1,\ldots,d$, PVMs on $\Hil$,}\\
\mbox{$(Q_j^{(l)})_{j=1}^m$, $l=1,\ldots,d$, PVMs on $\Hil$,}\\
\mbox{$[P_i^{(k)},Q_j^{(l)}]=0$ for all $i,j$ and $k,l$}
\end{array}
\},
\]
which looks more similar to $\cQ_c^n$. As observed in \cite{fritz,jungeetal},
Tsirelson's Problem is a problem about the \spc
\[
\fF=\ell_\infty^m*\cdots*\ell_\infty^m,
\]
the $d$-fold unital free product of $\ell_\infty^m$ (see Example~\ref{exa:freeprod}).
This algebra is also $*$-isomorphic to the group algebra $\IC[\G]$ of $\G=\IZ_m^{*d}$,
the $d$-fold free product of the finite cyclic group $\IZ_m$.
We denote by $(p_i)_{i=1}^m$ the standard basis of minimal projections in $\ell_\infty^m$,
and by $(p_i^{(k)})_{i=1}^m$ the $k$-th copy of it in the free product $\fF$.
We also write $p_i^{(k)}$ for the elements $p_i^{(k)}\otimes 1$ in $\fF\otimes\fF$
and $q_j^{(l)}$ for $1\otimes p_j^{(l)}$.
Now, it is not too hard to see that
\[
\cQ_c^n=\{ \mb{ \p(p_i^{(k)}q_j^{(l)}) }_{\begin{subarray}{c}k,l\\ i,j\end{subarray}} : \p\in S_n(\fF\otimes^{\max}\fF)\}
\]
and
\[
\cQ_s^n=\{ \mb{ \p(p_i^{(k)}q_j^{(l)}) }_{\begin{subarray}{c}k,l\\ i,j\end{subarray}} : \p\in S_n(\fF\otimes^{\min}\fF)\}.
\]
Here we recall that $S_n(\cA)$ denotes the set of u.c.p.\ maps from $\cA$ into $\IM_n(\IC)$.
When $(m,d)=(2,2)$, the group $\IZ_2^{*2}$ is isomorphic to the infinite dihedral group $D_\infty$
and the universal representation
$\langle\left[\begin{smallmatrix} 0 & 1\\ 1 & 0 \end{smallmatrix}\right],
  \left[\begin{smallmatrix} 0 & z\\ z^* & 0 \end{smallmatrix}\right]\rangle\subset\IM_2(C(\IT))$
gives rise to an embedding $\mathrm{C}^* D_\infty\subset\IM_2(C(\IT))$.
In particular, $\G$ is amenable and $\cQ_c^n=\cQ_s^n$ in this case.
Otherwise, $\IZ_m^{*d}$ is a ``nontrivial'' free product group and contains a copy of $\IF_2$.
It is proved in \cite{fritz,jungeetal} that the matricial version of Tsirelson's Problem
(conjectures $(\ref{cond:xmatsi1})$ and $(\ref{cond:xmatsi2})$ in the following theorem) is
equivalent to CEC and KC.
Here we prove that the original Tsirelson Problem is also equivalent to them.

\begin{thm}[cf.\ \cite{fritz,jungeetal}]
The following conjectures are equivalent.
\begin{enumerate}[$(1)$]
\item\label{cond:xcec}
The Connes Embedding Conjecture holds true.
\item\label{cond:xmatsi1}
One has $\cQ_c^n=\cQ_s^n$ for all $m,d,n\in\IN$.
\item\label{cond:xmatsi2}
There are $m,d\geq2$ with $(m,d)\neq(2,2)$ such that $\cQ_c^n=\cQ_s^n$ for all $n$.
\item\label{cond:xtsi}
One has $\cQ_c=\cQ_s$ for all $m,d\in\IN$.
\end{enumerate}
\end{thm}
\begin{proof}
The implication $(\ref{cond:xcec})\Rightarrow(\ref{cond:xmatsi1})$
follows from the fact that KC is equivalent to
that $\fF\otimes^{\max}\fF=\fF\otimes^{\min}\fF$ (see Section~\ref{sec:kirchberg}).
That $(\ref{cond:xmatsi1})\Rightarrow(\ref{cond:xmatsi2})\,\&\,(\ref{cond:xtsi})$ is obvious.
Now, if (\ref{cond:xmatsi2}) holds true, then the operator system structures on
\[
\lh\{p_i^{(k)}q_j^{(l)} : i,j,k,l\}
 = (\ell_\infty^m+\cdots+\ell_\infty^m)\otimes(\ell_\infty^m+\cdots+\ell_\infty^m)
 \subset \fF\otimes\fF
\]
induced from $\max$ and $\min$ tensor products coincide by Corollary~\ref{cor:ce}.
This implies KC by Corollary~\ref{cor:pisier}.
This proves $(\ref{cond:xmatsi2})\Rightarrow(\ref{cond:xcec})$.
Finally, it remains to show $(\ref{cond:xtsi})\Rightarrow(\ref{cond:xcec})$.
Let a finite von Neumann algebra $(M,\tau)$ and a sequence $u_1,u_2,\ldots$
of unitary elements be given.
By spectral theorem, for every $m$, there are unitary elements $u_i(m)$ of
order $m$ such that $\| u_i - u_i(m)\|\le\pi/m$.
Recall that $\fF\cong \IC[\IZ_m^{*d}]$
via $\sum_{i=1}^m \omega_m^{i}p_i^{(k)}\leftrightarrow s_k$, where
$\omega_m$ is an $n$-th primitive root of unity and $s_k$ is the generator
of the $k$-th copy of $\IZ_m$ in $\IZ_m^{*d}$.
Then, condition $(\ref{cond:xtsi})$ for $m,d$ implies that the map
\[
s_i\otimes s_j\mapsto \tau(u_i(m)u_j(m)^*)=\ip{\pi_\tau(u_i(m))\pi_\tau^{\op}((u_j(m)^*)^{\op})\hat{1},\hat{1}}
\]
is contractive on
$\lh\{ s_i\otimes s_j : i,j\le d\}\subset \mathrm{C}^*\IZ_m^{*d} \otimes_{\min} \mathrm{C}^*\IZ_m^{*d}$.
It is also contractive on $\mathrm{C}^*\IF_d \otimes_{\min} \mathrm{C}^*\IF_d$ through
the natural surjection $\IF_d\to\IZ_m^{*d}$.
Taking limit in $m$ and $d$, one verifies that
the condition in Theorem~\ref{thm:cecmin} holds.
It follows that $(M,\tau)$ satisfies CEC.
(Playing around with the fact that every element in a $\mathrm{II}_1$-factor is
a linear combination of projections with control on the number of terms and coefficients
(\cite{fh}), one can weaken condition $(\ref{cond:xtsi})$ further to the extent
that $\cQ_c=\cQ_s$ for a fixed $m\geq2$ and all $d\in\IN$.)
\end{proof}
\subsection{Quantum correlation matrices and semidefinite program}
We continue the study of Tsirelson's Problem whether $\cQ_c=\cQ_s$.
We will see that $\cQ_s\subset\cQ_c$ are approximated from above and from
below by rather explicit semi-algebraic sets. Recall that a subset of $\IR^N$
is said to be \emph{semi-algebraic} if it can be defined by finitely many
polynomial equations and inequalities.
It is very likely that $\cQ_c$ and $\cQ_s$ themselves are not semi-algebraic
(cf.\ Problem 2.10 in \cite{tsirelson}),
except the case $(m,d)=(2,2)$ where $\cQ_c=\cQ_s$ is semi-algebraic (since
the irreducible representations of $\IZ_2^{*2}$ are at most two dimensional).
For the description of $\cQ_c$, it is probably neater to use the isomorphism
$\ell_\infty^m\cong\IC[\IZ_m]$, given by the Fourier transform
$p_i\leftrightarrow \frac{1}{m}\sum_{v=1}^m\omega_m^{-iv}s^v$, where
$\omega_m$ is an $n$-th primitive root of unity and $s$ is the generator of $\IZ_m$.
Now, $\fF\otimes^{\max}\fF$ is identified with $\IC[\IZ_m^{*d}\times\IZ_m^{*d}]$.
We write $s_k$ and $t_l$ for the generator of $\IZ_m^{(k)}\times\{1\}$ and $\{1\}\times\IZ_m^{(l)}$
in $\IZ_m^{*d}\times\IZ_m^{*d}$.
By Corollary~\ref{cor:posdeffunct}, there is a one-to-one correspondence between
a positive type function $f$ on $\IZ_m^{*d}\times\IZ_m^{*d}$ and
a state $\p_f$ on $\IC[\IZ_m^{*d}\times\IZ_m^{*d}]$, given by
$\p_f(g)=\sum_t f(t)g(t)$.
Let $E_n$ be an increasing and exhausting sequence
of finite subsets of $\IZ_m^{*d}\times\IZ_m^{*d}$.
We say a function $f\colon E_n^{-1}E_n\to\IC$ is positive type on $E_n$
if $[f(s^{-1}t)]_{s,t\in E_n}$ is positive semidefinite in $\IM_{E_n}(\IC)$.
Then by compactness, one has
\begin{align*}
\cQ_c &= \bigcap_n \{ \mb{ \p_f(p_i^{(k)} q_j^{(l)})
 }_{\begin{subarray}{c}k,l\\ i,j\end{subarray}} : f \mbox{ positive type on $E_n$}\}\\
&= \bigcap_n \{ \mb{ \frac{1}{m^2}\sum_{v,w=1}^m \omega_m^{-iv-jw} f(s_k^v t_l^w)
 }_{\begin{subarray}{c}k,l\\ i,j\end{subarray}} : f \mbox{ positive type on $E_n$}\}.
\end{align*}
See \cite{fritz} for more information.
Next, we deal with $\cQ_s$. Although the description
\[
\cQ_s=\mathop{\mathrm{closure}}\left( \bigcup_n
\{ \mb{\ip{(P_i^{(k)}\otimes Q_j^{(l)})\xi,\xi}}_{\begin{subarray}{c}k,l\\ i,j\end{subarray}} :
\begin{array}{c}
\mbox{$\xi\in\ell_2^n\otimes\ell_2^n$ a unit vector,}\\
\mbox{$(P_i^{(k)})_{i=1}^m$ PVMs on $\ell_2^n$,}\\
\mbox{$(Q_j^{(l)})_{j=1}^m$ PVMs on $\ell_2^n$}
\end{array}
\}\right)
\]
is already quite good, we will do some calculations around operator system tensor products.
See \cite{fkpt} (and also \cite{fp,kavruk,kptt}) for more information on this subject.
Let
\[
\cS_{m,d}=\lh\{ p_i^{(k)} : i=1,\ldots,m,\,k=1,\ldots,d\}=\ell_\infty^m+\cdots+\ell_\infty^m\subset\fF
\]
be the finite-dimensional operator system.
Then, $\cQ_s$ coincides with the space of states
on $\cS_{m,d}\otimes^{\min}\cS_{m,d}$, evaluated at
$\{ p^{(k)}_i\otimes q^{(l)}_j \}$.
By Lemma~\ref{lem:kavruk}, the operator system structure of $\cS_{m,d}$
is defined as follows: $a\in\IM_n(\cS_{m,d})$ is positive if and only if
there are $a_1,\ldots,a_d\in\IM_n(\ell_\infty^m)_+$ such that $a=a_1+\cdots+a_m$.
This implies that a map $\p\colon\cS_{m,d}\to\IM_n(\IC)$ is c.p.\ if and only if
the restriction of $\p$ to each summand $\ell_\infty^m$ is c.p.
Meanwhile, a map $\psi\colon\ell_\infty^m\to\IM_n(\IC)$ is c.p.\ if and only if
$\psi(p_i)\geq0$ for every $i=1,\ldots,m$.
It follows that the dual operator system $\cS_{m,d}^\rd$ of $\cS_{m,d}$ is naturally
identified with
\[
\cS_{m,d}^\rd=\{ (f_1,\ldots,f_d) \in \ell_\infty^m\oplus\cdots\oplus\ell_\infty^m :
 \sum_i f_k(i)=\sum_j f_l(j)\} \subset \ell_\infty^{dm}.
\]
By Theorem~\ref{thm:fp}, one has
$(\cS_{m,d}\otimes^{\min}\cS_{m,d})^\rd=\cS_{m,d}^\rd\otimes^{\max}\cS_{m,d}^\rd$.
Hence, if $c\in (\cS_{m,d}\otimes^{\min}\cS_{m,d})^\rd$ is strictly positive,
i.e.\ $c\geq\e1$ for some $\e>0$,
then it has a representation $x^*(s\otimes t)x$ for some $x\in\IM_{vw,1}(\IC)$,
$s\in\IM_v(\cS_{m,d}^\rd)_+$, and $t\in\IM_w(\cS_{m,d}^\rd)_+$.
Viewing $x$ as a vector $\xi$ in $\ell_2^v\otimes\ell_2^w$, one has
\[
x^*(s\otimes t)x = \ip{(s\otimes t) \xi,\xi} \in \cS_{m,d}^\rd\otimes\cS_{m,d}^\rd.
\]
Evaluating it at the basis $\{ p^{(k)}_i\otimes q^{(l)}_j \}$, one obtains
\[
x^*(s\otimes t)x = \mb{ \ip{(s^{(k)}_i\otimes t^{(l)}_j)\xi,\xi}}_{\begin{subarray}{c}k,l\\ i,j\end{subarray}},
\]
where $s^{(k)}_i\in\IM_v(\IC)_+$ and $t^{(l)}_j\in\IM_w(\IC)_+$ satisfy $\sum_i s^{(k)}_i = a$ (independent of $k$)
and $\sum_j t^{(l)}_j=b$ (independent of $l$).
Replacing $\xi$ with $(a^{1/2}\otimes b^{1/2})\xi$, one may assume that $a=1$ and $b=1$.
By dilating them, one may further assume that $(s^{(k)}_i)_{i=1}^m$ and $(t^{(l)}_j)_{j=1}^m$
are PVMs.
In this way, we come back to the original definition of $\cQ_s$.
We note that we have proved that one can realize a ``generic" element $[\gamma_{i,j}^{k,l}] \in \cQ_s$
by a finite-dimensional system, namely if it is faithful on $\cS_{m,d}\otimes^{\min}\cS_{m,d}$.
It seems unlikely that this is also the case for all $[\gamma_{i,j}^{k,l}] \in \cQ_s$.

\end{document}